\documentclass{article}

\usepackage{graphicx}
\usepackage[utf8]{inputenc}
\usepackage{amsfonts, amsmath, amssymb, amsthm}
\usepackage[letterpaper, portrait, margin=1in]{geometry}
\usepackage{verbatim}
\usepackage{enumitem}
\usepackage{xspace}
\usepackage{tikz}
\usepackage{wrapfig}

\newtheorem{thmain}{Theorem}

\newtheorem{theorem}{Theorem}
\newtheorem{proposition}[theorem]{Proposition}

\newtheorem{definition}[theorem]{Definition}
\newtheorem{corollary}[theorem]{Corollary}
\newtheorem{lemma}[theorem]{Lemma}
\newtheorem{remark}[theorem]{Remark}
\newtheorem{conjecture}{Conjecture}

\newcommand{\keyterm}[1]{\emph{#1}}

\newcommand{\makenest}[3]{\newcommand{#1}[1]{\ensuremath{\left#2##1\right#3}}}
\makenest{\abs}||
\makenest{\norm}\lVert\rVert
\makenest{\parens}()
\makenest{\bracket}[]
\makenest{\set}\{\}
\makenest{\chev}<>
\makenest{\clop}[)
\makenest{\opcl}(]
\makenest{\eval}.|

\newcommand{\deriv}[2]{\frac{d#1}{d#2}}
\newcommand{\pderiv}[2]{\frac{\partial#1}{\partial#2}}

\newcommand{\textand}{\qquad \text{and} \qquad}

\newcommand{\R}{\mathbb{R}}
\newcommand{\Z}{\mathbb{Z}}
\newcommand{\N}{\mathbb{N}}

\newcommand{\Oo}{\mathcal{O}}
\newcommand{\dd}{\,d}
\DeclareMathOperator{\id}{id}

\newcommand{\cspace}{\mathcal{M}}
\newcommand{\cmap}{F}

\newcommand{\poly}[1]{\mathcal{P}_{#1}} % Classes of polynomials
\newcommand{\coef}[3]{\bracket{#1}_{#2}^{#3}} % Coefficient for given p,q,i,j
\newcommand{\coefp}[2]{\bracket{#1}_{#2}} % Coefficient for given p,q, as a polynomial in m and n
\newcommand{\hc}[2]{f_{#1}^{#2}} % Series for h_A
\newcommand{\lc}[2]{g_{#1}^{#2}} % Series for l_{m,n}

\title{On Length Spectrum Rigidity of Dispersing Billiard Systems}
\author{Otto Vaughn Osterman}
\date{August 25, 2022}

\begin{document}

\maketitle

\begin{abstract}
    We consider the class of dispersing billiard systems in the plane formed by removing three convex analytic scatterers satisfying the non-eclipse condition. The collision map in this system is conjugated to a subshift, providing a natural labeling of periodic points. We study the problem of marked length spectrum rigidity for this class of systems. We show that two such systems have the same marked length spectrum if and only if their collision maps are analytically conjugate to each other near a homoclinic orbit and that two scatterers and the marked length spectrum together uniquely determine the third scatterer.
\end{abstract}

\section{Introduction}

A dynamical billiard system is a dynamical system modeling the behavior of a point particle travelling in a region $D$, the \keyterm{billiard table}, undergoing elastic collisions with the boundary $ \partial D $. The collision map of a billiard system is the map from a position and angle of incidence at one collision point to the position and angle of incidence at the next collision point.

The \keyterm{perimeter} or \keyterm{length} of a periodic orbit is the total distance that a particle travels on the orbit in one period.
The \keyterm{length spectrum} of the system is the set of all positive real numbers that are perimeters of some periodic orbit. The problem of length spectrum rigidity is that of determining whether or not, in some class of billiard systems, the length spectrum uniquely defines the billiard table up to isometry.
In many classes of systems, there is a natural mapping from some set $\Sigma$ to the periodic points. In this case, the \keyterm{marked length spectrum} is the map that sends an element in $\Sigma$ to the perimeter of its respective periodic orbit.
The problem of marked length spectrum rigidity asks whether the marked length spectrum uniquely determines the billiard table.
Some results on this problem are already known for certain classes of billiard systems. For example:
\begin{itemize}
    \item In the class of billiard systems where the billiard table $D$ is a bounded convex domain with two axes of reflexive symmetry and a stable minor axis, the marked length spectrum uniquely determines $D$ up to isometry \cite{Verdiere}.
    \item In the class of billiard systems where $ D = \R^2 \setminus (D_1 \cup D_2 \cup \hdots \cup D_n) $ for convex analytic scatterers $ D_1,D_2,\hdots,D_n $ satisfying the non-eclipse condition such that the region $ D_1 \cup D_2 $ has two axes of symmetry, the marked length spectrum uniquely determines almost every system \cite{DKL}.
    \item If $D$ is a bounded strictly convex domain with an axis of symmetry that is sufficiently close to a circle in the $C^9$ topology, then any deformation of $D$ that preserves the length spectrum must be an isometric transformation \cite{DKW}.
\end{itemize}

Length spectrum rigidity problems in dynamical billiards have applications beyond billiard dynamics, most notably, to the well-known problem posed by Kac, ``Can you hear the shape of a drum?,'' which asks whether or not a domain in $\R^2$ is uniquely determined by its Laplace spectrum \cite{Kac}.
The Poisson relation for manifolds with boundary asserts that the wave trace of the Laplace operator is contained in the length spectrum. This result was originally proven for convex domains in \cite{AM}, but the arguments of \cite{MS} extend this to general domains (see \cite{PS}).
In fact, for generic strictly convex domains, the wave trace of the Laplace operator is equal to the length spectrum \cite{PS}. This implies that the Laplace spectrum uniquely determines the length spectrum, and thus that length spectrum rigidity implies generic Laplace spectrum rigidity.
For a more detailed discussion, see \cite{PS2}.

A dual problem to the inverse spectral problem suitable for unbounded domains is the inverse resonance problem, which asks whether or not a domain in $\R^2$ is uniquely determined up to isometry by its set of resonance poles, or singularities of the resolvent operator $ (\Delta - z^2)^{-1} $. For example, in \cite{Zelditch3}, Zelditch proves that generically, a domain that is the complement to two bounded, simply connected, analytic domains in $\R^2$ with a reflexive symmetry mapping one domain to the other is uniquely determined by its resonance poles. This proof is based on the methods in \cite{Zelditch1,Zelditch2} for calculating wave trace invariants around periodic orbits of a billiard.
A related result is proven in \cite{ISZ}, where the authors show that generically, given the spectral invariants of a semi-classical Schr\"{o}dinger operator, one can recover the full quantum Birkhoff normal form of its Hamiltonian, and in particular, the Birkhoff normal form of the billiard map.

In this paper, we consider the class of billiard systems arising from billiard tables of the form $ \R^2 \setminus (D_1 \cup D_2 \cup D_3) $ where the \keyterm{scatterers} $D_1$, $D_2$, and $D_3$ are pairwise disjoint, compact, convex, and have analytic boundaries with everywhere positive curvature. We refer to $D_1$, $D_2$, and $D_3$ as scatterers $1$, $2$, and $3$, respectively.
We also require that the system satisfies the \keyterm{non-eclipse condition}, i.e., there is no line in the plane that passes through all three scatterers.
Billiard systems where the billiard table is a complement of a union of open convex regions are called \keyterm{dispersing} and are known to have strong hyperbolic and chaotic properties. Thus, it is natural to expect an abundance of periodic orbits and for these periodic orbits to give a great amount of information on the system itself.

We parametrize $ \partial D_i $ by a closed curve $ \gamma_i: [0,\pi_i] \rightarrow \R^2 $ with arc length parameter and counter-clockwise orientation, where $\pi_i$ is the perimeter of scatterer $D_i$. We will denote a point in the phase space $\cspace$ of the collision map $\mathcal{\cmap}$ by $ (i,s,r) $, where:
\begin{itemize}
    \item The collision occurs on scatterer $i$ at the point $\gamma_i(s)$, and
    \item $ r = \sin\varphi $, where $\varphi$ is the signed angle of incidence of the collision.
\end{itemize}
We denote the connected components of the phase space by $ \cspace_i = \set{i} \times \R/\pi_i\Z \times [-1,1] $. We denote the restriction of the collision map to $ \cspace_i \cap \mathcal{\cmap}^{-1}(\cspace_j) $ by $\cmap_{ij}$. Often in this case we will write $ \cmap_{ij} : (s,r) \mapsto (s',r') $ instead of $ \cmap_{ij}: (i,s,r) \mapsto (j,s',r') $ since $i$ and $j$ are fixed by this restriction.
The collision map has a time-reversing involution $ I_0: (i,s,r) \mapsto (i,s,-r) $ with
\begin{equation} \label{involution}
    I_0 \circ \cmap = \cmap^{-1} \circ I_0.
\end{equation}

We assume without loss of generality that $\gamma_1(0)$ and $\gamma_2(0)$ are the points of shortest distance between $ \partial D_1 $ and $ \partial D_2 $. We can apply an isometric transformation to the system so that $ \gamma_1(0) = \parens{0, \frac{\ell_0}{2}} $ and $ \gamma_2(0) = \parens{0, -\frac{\ell_0}{2}} $ for some $ \ell_0 > 0 $, and we do so. By the non-eclipse condition, we may also assume that scatterer $3$ lies entirely to the right of the vertical axis.

\begin{definition}
    We define the \keyterm{coding sequence} of a point $ x \in \cspace $ to be the sequence $ \set{i_k}_{k \in \Z} $ such that $ \cmap^k(x) \in \cspace_{i_k} $.
    If $x$ is $n$-periodic, then its coding sequence is also $n$-periodic, so in this case, we sometimes say that the coding sequence of $x$ is the finite sequence $\set{i_k}_{k=0}^{n-1}$.
\end{definition}

We restrict our consideration to the non-escaping Cantor set
$$ \mathcal{C} := \bigcap_{k \in \Z} \cmap^{-k}(\cspace). $$
As proven in \cite{BDKL,DKL}, given the non-eclipse condition, the coding sequence is a bijective correspondence between $\mathcal{C}$ and the sequences $ \set{i_k}_{k \in \Z} \in \set{1,2,3}^\Z $ such that $ i_k \neq i_{k+1} $ for all $k$.
In particular, if we define
$$ \Sigma := \set{(i_0,i_1,i_2,\hdots,i_{n-1}): n \in \N, i_{k-1} \neq i_k \text{ for all } k, \text{ and } i_0 \neq i_{n-1}}, $$
then the periodic orbits are in bijective correspondence with $\Sigma$.
We therefore define the \keyterm{marked length spectrum} of such a billiard system to be the map $ \mathcal{L}: \sigma \mapsto L(\sigma) $, where $L(\sigma)$ is the perimeter of the periodic orbit with coding $\sigma$.
In this paper, we are concerned with the problem of marked length spectrum rigidity, i.e., whether or not the system itself is uniquely determined up to isometry by its marked length spectrum.

We primarily work with the cyclicity-$2$ periodic orbits with coding $ 3(12)^{m-1}13(12)^{n-1}1 $. We define $x_0^{m,n}$ to be the periodic point with this coding and $ \set{x_k^{m,n}}_{k=0}^{2m+2n-1} $ to be the orbit of this point, so that $ x_0^{m,n}, x_{2m}^{m,n} \in \cspace_3 $. We write $ x_k^{m,n} = (i_k^{m,n}, s_k^{m,n}, r_k^{m,n}) $ and denote the perimeter of this orbit by $\ell_{m,n}$.

The point $x_{2m}^{m,n}$ is periodic with coding $ 3(12)^{n-1}13(12)^{m-1}1 $. By uniqueness of the periodic point with this coding, $ x_{2m}^{m,n} = x_0^{n,m} $, and in general,
\begin{equation} \label{interchange}
    x_{2m+k}^{m,n} = x_k^{n,m}.
\end{equation}
It follows from (\ref{involution}) that if $x$ is a periodic point, then $I_0(x)$ is also a periodic point whose coding sequence is the reversal of that of $x$. Therefore, $I_0(x_0^{m,n})$ has coding sequence $ 3(12)^{n-1}13(12)^{m-1}1 $, so $ I_0(x_0^{m,n}) = x_0^{n,m} $.
Making use of (\ref{involution}) and (\ref{interchange}), we have
\begin{equation} \label{involution-orbit}
    x_k^{m,n} = I_0(x_{2m-k}^{m,n})
\end{equation}
for all $k$.
In particular, $ x_m^{m,n} = I(x_m^{m,n}) $ and $ x_{2m+n}^{m,n} = I(x_{2m+n}^{m,n}) $, so these points are perpendicular bounces. If $m$ is even, then the perpendicular bounce $x_m^{m,n}$ will occur on scatterer $2$, while if $m$ is odd, it will occur on scatterer $1$. Similarly, the parity of $n$ determines whether the perpendicular bounce $x_{2m+n}^{m,n}$ occurs on scatterer $1$ or $2$.

A special case of these types of periodic orbits are the cyclicity-$1$ orbits with coding $ 3(12)^{n-1}1 $. We define $x_0^n$ to be the point with this coding on $\cspace_3$, $ \set{x_k^n}_{k=0}^{2n-1} $ to be the orbit of this point, and $\ell_n$ to be the perimeter of this orbit.
The cyclicity-$2$ orbit $ \set{x_k^{n,n}}_{k=0}^{4n-1} $ is exactly this orbit repeated twice, so $ \ell_n = \frac{1}{2} \ell_{n,n} $.

We denote the homoclinic point on $\cspace_3$ with coding $ (21)^\infty3(12)^\infty $ by $x_0^\infty$ and its orbit by $ \set{x_k^\infty}_{k \in \Z} $.
Then, $ x_k^\infty = I(x_{-k}^\infty) $, and for all $ k \in \Z $, $ x_k^{m,n} \rightarrow x_k^\infty $ and $ x_{2m+k}^{m,n} \rightarrow x_k^\infty $ as $ m,n \rightarrow \infty $.
We write $ x_k^\infty = (i_k^\infty, s_k^\infty,r_k^\infty) $.
This homoclinic orbit is illustrated in Figure \ref{fig:homoclinic}.

\begin{figure}[ht]
    \centering
    \begin{tikzpicture}
        \clip (-1.5,-5) rectangle (15,5.5);
        \fill[black!25] (0,9.4) ellipse (12.2 and 6.9);
        \fill[black!25] (0,-13.9) ellipse (13.3 and 11.4);
        \fill[black!25] (18,2.5) ellipse (5.5 and 4.1);
        \node at (3,4.2) {$D_1$};
        \node at (3,-4.1) {$D_2$};
        \node at (14,1.7) {$D_3$};
		
		% Draw period-2 orbit
		\draw[fill] (0,2.5) circle[radius=1.5pt];
		\draw[fill] (0,-2.5) circle[radius=1.5pt];
		\draw (0,2.5) -- (0,-2.5);
		\node[above] at (-.45,2.5) {$\gamma_1(0)$};
		\node[below] at (-.35,-2.5) {$\gamma_2(0)$};
        
        % Draw homoclinic orbit
		\draw[fill] (12.5413, 3.0017) circle[radius=1.5pt];
		\node[right](A0) at (12.5413, 3.0017) {$\gamma_3(s_0^\infty)$};
		\draw[fill] (7.2940, 3.8690) circle[radius=1.5pt];
		\node[above](A1) at (7.0440, 3.8690) {$\gamma_1(s_1^\infty)$};
		\draw[fill] (2.6911, -2.7358) circle[radius=1.5pt];
		\node[below](A2) at (2.6411, -2.7358) {$\gamma_2(s_2^\infty)$};
		\draw[fill] (1.2973, 2.5391) circle[radius=1.5pt];
		\node[above](A3) at (1.7473, 2.5391) {$\gamma_1(s_3^\infty)$};
		\draw[fill] (0.6002, -2.5116) circle[radius=1.5pt];
		\node[below](A4) at (0.8502, -2.5116) {$\gamma_2(s_4^\infty)$};
		\draw[fill] (0.2999, 2.5021) circle[radius=1.5pt];
		\node[above](A5) at (0.5499, 2.5021) {$\gamma_1(s_5^\infty)$};
		\draw (12.5413, 3.0017) -- (7.2940, 3.8690);
		\draw (7.2940, 3.8690) -- (2.6911, -2.7358);
		\draw (2.6911, -2.7358) -- (1.2973, 2.5391);
		\draw (1.2973, 2.5391) -- (0.6002, -2.5116);
		\draw (0.6002, -2.5116) -- (0.2999, 2.5021);
		\draw (0.2999, 2.5021) -- (0.2038, -0.4995);
    \end{tikzpicture}
    \caption{The homoclinic orbit $ \set{x_k^\infty}_{k \in \Z} $, with $ x_k^\infty = \parens{i_k^\infty, s_k^\infty, r_k^\infty} $, and the period-$2$ orbit $ \set{(1,0,0), (2,0,0)} $.}
    \label{fig:homoclinic}
\end{figure}
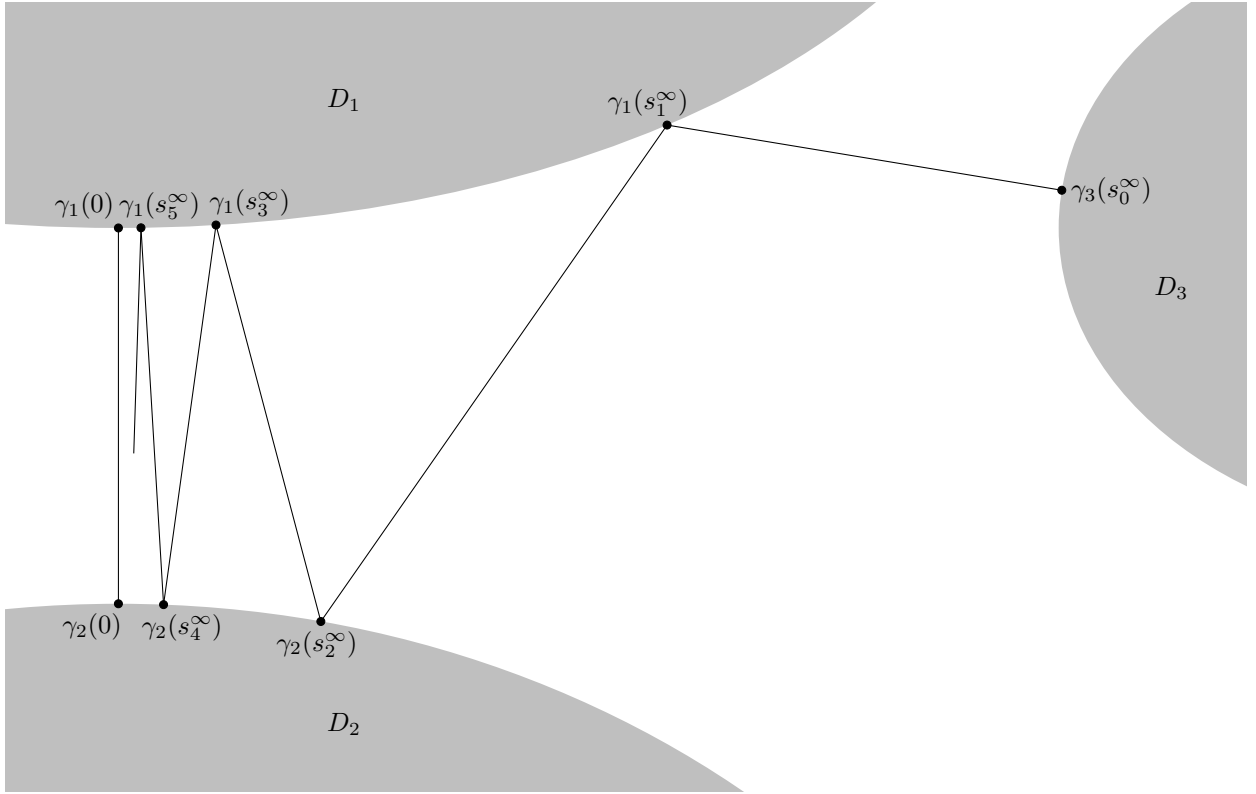

For $ i \in \set{1,2} $, we denote the stable manifold of the fixed point $(i,0,0)$ of $F^2$ by $W_i^s$ and the unstable manifold by $W_i^u$. Then, $ x_k^\infty \in W_1^s $ and $ x_{-k}^\infty \in W_1^u $ for all odd $ k > 0 $, and $ x_k^\infty \in W_2^s $ and $ x_{-k}^\infty \in W_2^u $ for all even $ k > 0 $.

In this paper, we prove the following:

\begin{thmain} \label{th:conj}
    Let $\mathcal{B}$ be the class of $3$-scatterer dispersing billiard systems in the plane satisfying the non-eclipse condition where the scatterers have analytic boundaries.
    Two systems in $\mathcal{B}$ have the same length spectral data $\ell_{m,n}$ if and only if there is a local analytic conjugacy between their collision maps defined in some neighborhood of the homoclinic orbit $ \set{x_k^\infty}_{k \in \Z} $.
\end{thmain}

\begin{thmain} \label{th:fixed}
   If two systems in $\mathcal{B}$ have the same length spectral data $\ell_{m,n}$ and the region $ D_1 \cup D_2 $ in one system is isometric to that in the other system, then the entire billiard tables are isometries of each other.
\end{thmain}

Theorems \ref{th:conj} and \ref{th:fixed} can be extended to dispersing billiard systems with more than three scatterers satisfying the non-eclipse condition in the same way as the results in \cite{DKL}, but we do not do this explicitly.

To prove Theorem \ref{th:conj}, we follow four major steps.
In Section \ref{sec:Birkhoff}, we show that the collision map can be conjugated to a Birkhoff normal form $N$ near the $2$-periodic orbit. We extend the conjugacy to a neighborhood of the entire homoclinic orbit by defining a gluing map $G$ near the homoclinic point $x_0^\infty$.
In Section \ref{sec:series}, we define a type of asymptotic power series and show that the Birkhoff energies of the cyclicity-$2$ orbits can be expressed as such approximations.
In Section \ref{sec:length}, we show how the perimeters of periodic orbits can be computed from their dynamics in normal form. In the process, we show that the length spectral data $\ell_{m,n}$ has a similar type of asymptotic series expansion.
In Section \ref{sec:recover}, we examine certain coefficients in this asymptotic series for $\ell_{m,n}$ more closely and relate them to certain coefficients in the power series expansions of $N$ and $G$. We show that this provides sufficient information to recover $N$ and $G$ completely, and thus completely determines the dynamics of the conjugated system in normal form.
As part of this proof, we will show that the length spectral data can be determined entirely from $N$ and $G$, and thus the converse statement in Theorem \ref{th:conj} is also true.

The proof of Theorem \ref{th:conj} is motivated by that in \cite{DKL}, where the authors recover the normal form $N$ and partial information on the gluing map $G$ from the length spectral data of the cyclicity-$1$ orbits.
We use many of the same methods, but we use the cyclicity-$2$ orbits rather than the cyclicity-$1$ orbits as we need to recover the function $G$ entirely.

In the case where the region $ D_1 \cup D_2 $ has two axes of reflexive symmetry, the normal form $N$ is uniquely determined by the length spectrum (see \cite{Verdiere, DKL}). However, in the general case, or even assuming one axis of symmetry, this cannot be true, and the converse statement of Theorem \ref{th:conj} suggests that one cannot determine the geometry of scatterers $1$ and $2$ near the $2$-periodic points from the length spectrum without also considering the behavior of the system elsewhere.
We discuss this in more detail in Section \ref{sec:conclusion}.

To prove Throrem \ref{th:fixed}, we use the fact that the region $ D_1 \cup D_2 $ uniquely determines the conjugacy of the normal form system to the original billiard system on $ \cspace_1 \cup \cspace_2 $. We then have knowledge of not only the dynamics of the orbits $\set{x_k^n}$ in the conjugated system in normal form, but also of the dynamics of these orbits in the billiard system for all points except one.
A short geometric argument recovers the remaining point and the scatterer $D_3$ uniquely.

\section{Dynamics in Normal Form} \label{sec:Birkhoff}

Since we are concerned with analyzing dynamics up to an analytic conjugacy, we first conjugate the dynamics into some normal form.
Some of the results in this section parallel those in \cite{DKL,KLZ}.
Recall the following:

\begin{theorem} \label{th:BNF}
    Let $U$ be an open subset of $\R^2$ containing the origin and $ F: U \rightarrow \R^2 $ be an area-preserving and orientation-preserving analytic map with a hyperbolic fixed point at $(0,0)$ such that the eigenvalues of $DF(0,0)$ are $\lambda$ and $\lambda^{-1}$ for some $ \lambda \in (0,1) $. Then, there is a unique analytic function $\mu$ such that $F$ is conjugate to the \keyterm{Birkhoff normal form}
    \begin{equation} \label{N-def}
        N: (\xi,\eta) \mapsto \parens{\mu(\xi\eta) \xi, \frac{\eta}{\mu(\xi\eta)}}
    \end{equation}
    by an analytic area-preserving and orientation-preserving function $\Phi$, i.e., $ \Phi \circ N = F \circ \Phi $ near the origin.
    Furthermore, if $\tilde{\Phi}$ is another analytic map such that $ \tilde{\Phi} \circ N = F \circ \tilde{\Phi} $, then $ \tilde{\Phi} = \Phi \circ \Psi $, where $\Psi$ takes the form
    \begin{equation} \label{Psi-BNF}
        \Psi: (\xi,\eta) \mapsto \parens{\tau(\xi\eta) \xi, \frac{\eta}{\tau(\xi\eta)}}
    \end{equation}
    for some analytic function $\tau$.
\end{theorem}

It was originally proven by Birkhoff \cite{Birkhoff} that any area-preserving function on $\R^2$ is locally conjugate to such a normal form near a hyperbolic fixed point where the derivative has positive eigenvalues. Analyticity of the conjugacy and the relationship between two conjugacies $\Phi$ and $\tilde{\Phi}$ was proven by Moser \cite{Moser}.
Note that the quantity $\xi\eta$, the \keyterm{Birkhoff energy}, is invariant under $N$.

We will prove that there is a local semi-conjugacy from $\cmap$ to a Birkhoff normal form. We call the image of a point under this semi-conjugacy its \keyterm{Birkhoff coordinates}.
We define an involution in Birkhoff coordinates, $ I: (\xi,\eta) \mapsto (\eta,\xi) $, so that
\begin{equation} \label{involution-Birkhoff}
    I \circ N = N^{-1} \circ I.
\end{equation}
We now show that we can conjugate the collision map near the $2$-periodic orbit to a Birkhoff normal form in a manner so that the time-reversing involution is $I$ in Birkhoff coordinates. 

\begin{proposition} \label{prop:BNF}
    There exist open neighborhoods $ U_1,U_2 \subseteq \R^2 $ of $(0,0)$, $ V_1 \subseteq \cspace_1 $ of $(1,0,0)$, and $ V_2 \subseteq \cspace_2 $ of $(2,0,0)$, a unique analytic function $ \mu: (-h_*, h_*) \rightarrow (0,1) $ for some $ h_* > 0 $, and analytic, area-preserving, and orientation-preserving diffeomorphisms $ \Phi_1: U_1 \rightarrow V_1 $ and $ \Phi_2: U_2 \rightarrow V_2 $ such that, for $N$ defined by (\ref{N-def}) and $\cmap_{12}$ and $\cmap_{21}$ defined in the introduction:
    \begin{enumerate}[label = (\alph*)]
        \item $ \mu(0) = \lambda $, where $\lambda^2$ is the smallest eigenvalue of $D\cmap^2(0,0)$
        \item $ \Phi_2 \circ N = \cmap_{12} \circ \Phi_1 $ everywhere on $ U_1 \cap N^{-1}(U_2) $
        \item $ \Phi_1 \circ N = \cmap_{21} \circ \Phi_2 $ everywhere on $ U_2 \cap N^{-1}(U_1) $
        \item $ \Phi_1 \circ I = I_0 \circ \Phi_1 $ everywhere on $U_1$
        \item $ \Phi_2 \circ I = I_0 \circ \Phi_2 $ everywhere on $U_2$
    \end{enumerate}
\end{proposition}

We first need the normal form of $F^2$, for which $(1,0,0)$ is a fixed point. The following lemma follows from Proposition \ref{prop:BNF} and is proven in \cite{DKL}:

\begin{lemma} \label{lem:F2-BNF}
    There is an analytic area-preserving and orientation-preserving map $\Phi_1$ and a unique Birkhoff normal form $N^2$ defined locally near the origin such that
    \begin{equation} \label{F2-BNF}
        \Phi_1 \circ N^2 = \cmap^2 \circ \Phi_1
    \end{equation}
    and $ \Phi_1 \circ I = I_0 \circ \Phi_1 $ locally near the fixed point $(1,0,0)$.
\end{lemma}

\begin{proof}[Proof of Proposition \ref{prop:BNF}]
    Let $\Phi_1$ and $N^2$ be as in Lemma \ref{lem:F2-BNF}. We then have $ \Phi_1 \circ N^2 = \cmap_{21} \circ \cmap_{12} \circ \Phi_1 $, since $ \cmap^2 = \cmap_{21} \circ \cmap_{12} $ near the fixed point $(1,0,0)$.
    The normal form $N$ is then the square root of $N^2$, defined by replacing $\mu(\xi\eta)$ with $\mu(\xi\eta)^{1/2}$ in (\ref{N-def}).
    Conditions (a) and (d) follow directly from Lemma \ref{lem:F2-BNF} and this construction.
    
    Define $ \Phi_2 := \cmap_{12} \circ \Phi_1 \circ N^{-1} $, so that condition (b) is automatically satisfied. The remaining identities can be proven algebraically. The identity (\ref{F2-BNF}) and condition (b) imply
    $$ \Phi_1 \circ N = \Phi_1 \circ N^2 \circ N^{-1} = \cmap_{21} \circ \cmap_{12} \circ \Phi_1 \circ N^{-1} = \cmap_{21} \circ \Phi_2, $$
    proving (c). The identities (b), (\ref{involution-Birkhoff}), (d), (\ref{involution}), and (c) imply
    $$ \Phi_2 \circ I = \cmap_{12} \circ \Phi_1 \circ N^{-1} \circ I = \cmap_{12} \circ \Phi_1 \circ I \circ N = \cmap_{12} \circ I_0 \circ \Phi_1^{-1} \circ N = I_0 \circ \cmap_{21}^{-1} \circ \Phi_1 \circ N = I_0 \circ \Phi_2, $$
    proving (e).
\end{proof}

For any point $ x \in \cspace_1 $ or $ x \in \cspace_2 $, we call the point $ \Phi_1^{-1}(x) $ or $ \Phi_2^{-1}(x) $, respectively, the \keyterm{Birkhoff coordinates} of $x$.
The following lemma is analogous to the argument in Section 3 of \cite{DKL}.

\begin{lemma} \label{lem:extension}
    For some $ n_0 \in \N $, the sets $U_1$, $U_2$, $V_1$, and $V_2$ in Proposition \ref{prop:BNF} can be extended so that $V_1$ contains all points in the orbits $\set{x_k^{m,n}}$ in $ \cspace_1 \cup \cspace_2 $ for $ m,n \geq n_0 $, and statements (a)-(e) in Proposition \ref{prop:BNF} remain satisfied.
\end{lemma}

\begin{proof}
    Consider any $ (\xi,\eta) \in U_1 $. If $ (\cmap_{12} \circ \Phi_1)(\xi,\eta) $ is defined (that is, if $ F\parens{\Phi_1(\xi,\eta)} $ is defined and in $\cspace_2$), then we may define $ \Phi_2\parens{N(\xi,\eta)} = (F \circ \Phi_1)(\xi,\eta) $. Similarly, if $ (\cmap_{12}^{-1} \circ \Phi_1)(\xi,\eta) $ is defined, then we may define $ \Phi_2\parens{N^{-1}(\xi,\eta)} = (F^{-1} \circ \Phi_1)(\xi,\eta) $. In this manner, we may extend $U_2$ and $V_2$ so that $ V_2 \subseteq F_{12}(V_1) \cup F_{12}^{-1}(V_1) $.
    It can easily be proven that these definitions agree with the already-defined values of $\Phi_1\parens{N(\xi,\eta)}$ and $\Phi_1\parens{N^{-1}(\xi,\eta)}$ if $ N(\xi,\eta) \in U_1 $ or $ N^{-1}(\xi,\eta) \in U_1 $, and that all conditions (a)-(e) in Proposition \ref{prop:BNF} remain satisfied under this extension.
    By a similar process, we can extend $U_1$ and $V_1$ so that $ V_1 \subseteq \cmap_{21}(V_2) \cup \cmap_{21}^{-1}(V_2) $.
    We extend the sets $U_1$, $U_2$, $V_1$, and $V_2$ as much as possible in this manner, so that
    \begin{equation} \label{extension}
        V_1 \subseteq \cmap_{21}(V_2) \cup \cmap_{21}^{-1}(V_2) \textand V_2 \subseteq \cmap_{12}(V_1) \cup \cmap_{12}^{-1}(V_1).
    \end{equation}
    
    Since $ x_{2j}^\infty \rightarrow (2,0,0) $ as $ j \rightarrow \infty $ and $V_2$ is a neighborhood of $(2,0,0)$, we may take $j$ so that $ x_{2j}^\infty \in V_2 $. Since $ x_{2j}^{m,n} \rightarrow x_{2j}^\infty $ as $ m,n \rightarrow \infty $, there is some $ n_0 \in \N $ such that $ x_{2j}^{m,n} \in V_2 $ for all $ m,n \geq n_0 $.
    By repeated application of (\ref{extension}), it follows that $ x_k^{m,n} \in V_1 \cup V_2 $ for $ 1 \leq k \leq 2m-1 $. Similarly, $ x_{2j}^{n,m} \in V_2 $, so $ x_{2m+k}^{m,n} = x_k^{n,m} \in V_1 \cup V_2 $ for $ 1 \leq k \leq 2n-1 $.
    This includes all points in the orbit $\set{x_k^{m,n}}$ in $ \cspace_1 \cup \cspace_2 $, so the proof is complete.
\end{proof}

There is some $ \xi_\infty \neq 0 $ such that the Birkhoff coordinates of $ x_k^\infty $ for $ k \geq 1 $ is $ (\lambda^k \xi_\infty, 0) $. The Birkhoff coordinates of $ x_{-k}^\infty $ are then $ (0, \lambda^k \xi_\infty) $.
All of our results remain valid if we replace $\Phi_1$ and $\Phi_2$ with the maps $ \Phi_1 \circ (-\id) $ and $ \Phi_2 \circ (-\id) $, respectively, so we may assume without loss of generality that $ \xi_\infty > 0 $, and we do so.

Lemma \ref{lem:extension} implies that there is some neighborhood $U_-$ of $ (\xi_\infty, 0) $ on which the map $ \Phi_- := F_{13}^{-1} \circ \Phi_1 \circ N $ is defined and a neighborhood $U_+$ of $ (0, \xi_\infty) $ on which the map $ \Phi_+ := F_{13} \circ \Phi_1 \circ N^{-1} $ is defined.
It follows from these definitions and Proposition \ref{prop:BNF}(d) that
\begin{equation} \label{Phi-3-involution}
    \Phi_+ \circ I = I_0 \circ \Phi_-.
\end{equation}
By replacing $U_-$ and $U_+$ with smaller neighborhoods if necessary, we can assume that $\Phi_-$ and $\Phi_+$ are both diffeomorphisms onto some neighborhood $ V_3 \subseteq \cspace_3 $ of $x_0^\infty$.
This means points in $V_3$ will have two different Birkhoff coordinates. For example, the homoclinic point $x_0^0$ has Birkhoff coordinates $(\xi_\infty,0)$ and $(0,\xi_\infty)$.
We identify these points with each other via the following:

\begin{definition} \label{def:G}
    We define the \keyterm{gluing map} $ G: U_+ \rightarrow U_- $ by $ G := \Phi_-^{-1} \circ \Phi_+ $.
\end{definition}

Lemma \ref{lem:extension} immediately implies the following:

\begin{proposition} \label{prop:orbit-conj-N}
    For all sufficiently large $m$ and $n$, all points of the periodic orbit $\set{x_k^{m,n}}$ are contained in $ V := V_1 \cup V_2 \cup V_3 $.
\end{proposition}

For $m$ and $n$ large enough so that Proposition \ref{prop:orbit-conj-N} holds, we define $ (\xi_{m,n}, \eta_{m,n}) := \Phi_-^{-1}(x_0^{m,n}) $, and for $ k \neq 0,2m $, $ (\xi_k^{m,n}, \eta_k^{m,n}) $ to be the Birkhoff coordinates of $x_k^{m,n}$.
When the values of $m$ and $n$ are clear from context, we will often denote $ (\xi_A, \eta_A) := (\xi_{m,n}, \eta_{m,n}) $ and $ (\xi_B, \eta_B) := (\xi_{n,m}, \eta_{n,m}) $. Hence, $ (\xi_A, \eta_A) $ is the unique fixed point of $ G \circ N^{2n} \circ G \circ N^{2m} $ and $ (\xi_B, \eta_B) $ is the unique fixed point of $ G \circ N^{2m} \circ G \circ N^{2n} $.
By (\ref{interchange}), $ (\xi_{n,m}, \eta_{n,m}) = \Phi_-^{-1}(x_{2m}^{m,n}) $, so we may interchange the subscripts $A$ and $B$ simply by interchanging $m$ and $n$.

We denote the Birkhoff energies of the two cycles of the orbit by $ h_{m,n} := h_A := \xi_A \eta_A $ and $ h_{n,m} := h_B := \xi_B \eta_B $. Then, $ \xi_k^{m,n} \eta_k^{m,n} = h_A $ for $ 1 \leq k \leq 2m-1 $ and $ \xi_k^{m,n} \eta_k^{m,n} = h_B $ for $ 2m+1 \leq k \leq 2m+2n-1 $.
By (\ref{involution-orbit}), it holds that
\begin{equation} \label{orbit-symmetry}
    N^{2m}(\xi_A,\eta_A) = (\eta_A,\xi_A) \textand N^{2n}(\xi_B,\eta_B) = (\eta_B,\xi_B),
\end{equation}
that is, $ G \circ I $ maps the points $(\xi_A,\eta_A)$ and $(\xi_B,\eta_B)$ to each other.
In fact, it follows from Definition \ref{def:G} and equation (\ref{Phi-3-involution}) that $ G \circ I $ is an involution, so if $ (s,r) = \Phi_-(\xi_A,\eta_A) $, then $ (s,-r) = \Phi_-(\xi_B,\eta_B) $.
The conjugacy between the periodic orbits of the collision map $\cmap$ and their dynamics in Birkhoff coordinates is illustrated in Figure \ref{fig:conj-orbit}.

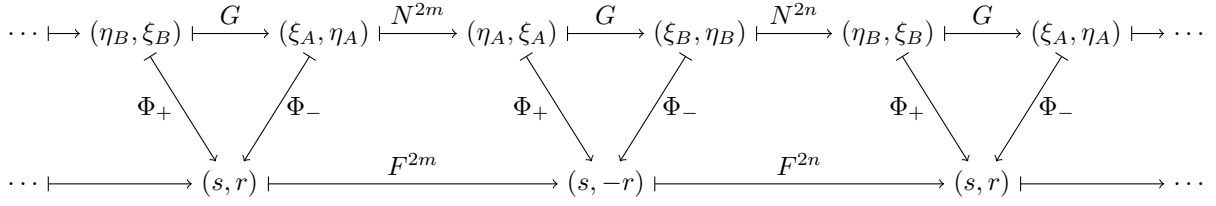
\begin{figure}[ht]
    \centering
    \begin{tikzpicture}
        \node (A1) at (-4,2) {$\hdots$};
        \node (A2) at (-2.5,2) {$(\eta_B,\xi_B)$};
        \node (A3) at (0,2) {$(\xi_A,\eta_A)$};
        \node (A4) at (2.5,2) {$(\eta_A,\xi_A)$};
        \node (A5) at (5,2) {$(\xi_B,\eta_B)$};
        \node (A6) at (7.5,2) {$(\eta_B,\xi_B)$};
        \node (A7) at (10,2) {$(\xi_A,\eta_A)$};
        \node (A8) at (11.5,2) {$\hdots$};
        \node (B1) at (-4,0) {$\hdots$};
        \node (B3) at (-1.25,0) {$(s,r)$};
        \node (B5) at (3.75,0) {$(s,-r)$};
        \node (B7) at (8.75,0) {$(s,r)$};
        \node (B8) at (11.5,0) {$\hdots$};
        
        \path[|->] (A1) edge (A2);
        \path[|->] (A2) edge node[above]{$G$} (A3);
        \path[|->] (A3) edge node[above]{$N^{2m}$} (A4);
        \path[|->] (A4) edge node[above]{$G$} (A5);
        \path[|->] (A5) edge node[above]{$N^{2n}$} (A6);
        \path[|->] (A6) edge node[above]{$G$} (A7);
        \path[|->] (A7) edge (A8);
        \path[|->] (B1) edge (B3);
        \path[|->] (B3) edge node[above]{$F^{2m}$} (B5);
        \path[|->] (B5) edge node[above]{$F^{2n}$} (B7);
        \path[|->] (A2) edge node[left]{$\Phi_+$} (B3);
        \path[|->] (A3) edge node[right]{$\Phi_-$} (B3);
        \path[|->] (A4) edge node[left]{$\Phi_+$} (B5);
        \path[|->] (A5) edge node[right]{$\Phi_-$} (B5);
        \path[|->] (A6) edge node[left]{$\Phi_+$} (B7);
        \path[|->] (A7) edge node[right]{$\Phi_-$} (B7);
        \path[|->] (B7) edge (B8);
    \end{tikzpicture}
    \caption{The conjugacy between a cyclicity-$2$ orbit of the collision map and the dynamics of the orbit in Birkhoff coordinates.}
    \label{fig:conj-orbit}
\end{figure}

For the remainder of this section, we consider $ (\xi_A,\eta_A), (\xi_B,\eta_B) \in U_- $ to be any points such that $ (\xi_B,\eta_B) = (G \circ I)(\xi_A,\eta_A) $, suspending the requirement that these be derived from a particular cyclicity-$2$ orbit.

\begin{lemma} \label{lem:coord}
    There is a local coordinate chart of $U_-$ near $(\xi_\infty,0)$ described by the coordinates $(\eta_A,\eta_B)$ and another local coordinate chart described by the coordinates $(h_A,h_B)$.
\end{lemma}

\begin{proof}
    For any point $ (\xi,\eta) \in U_- $, we may write $ (\xi_A,\eta_A) = (\xi,\eta) $ and $ (\xi_B,\eta_B) = (G \circ I)(\xi,\eta) $.
    Let $ W_3^s = \cmap^{-1}(W_1^s) \cap V_3 $ and $ W_3^u = \cmap(W_1^u) \cap V_3 $ be the stable and unstable manifolds near the homoclinic point. Then, $ \Phi_-^{-1}(W_3^s) = \set{\eta = 0} = \set{\eta_A = 0} $ and $ \Phi_-^{-1}(W_3^u) = G\parens{\Phi_+^{-1}(W_3^u)} = G(\set{\xi=0}) = (G \circ I)(\set{\eta=0}) = \set{\eta_B = 0} $.
    The curves $W^s$ and $W^u$ are transverse, and therefore, so are the curves $ \set{\eta_A = 0} $ and $ \set{\eta_B = 0} $.
    The collections of curves $ \set{\eta_A = x} $ and $ \set{\eta_B = y} $ are both analytic foliations, so the curves $ \set{\eta_A = x} $ and $ \set{\eta_B = y} $ are transverse to each other for $(x,y)$ near $(0,0)$. Hence, there is a local coordinate chart described by $(\eta_A,\eta_B)$.
    
    With this coordinate chart, we may write $ \xi_A = \psi_A(\eta_A,\eta_B) $ and $ \xi_B = \psi_B(\eta_A,\eta_B) $ for analytic functions $\psi_A$ and $\psi_B$. Define the function $ \Psi(\eta_A,\eta_B) = \parens{\eta_A \psi_A(\eta_A,\eta_B), \eta_B \psi_B(\eta_A,\eta_B)} $. Then,
    \begin{equation} \label{Psi-p-h}
        \Psi(\eta_A,\eta_B) = (h_A,h_B).
    \end{equation}
    Since $ \psi_A(0,0) = \psi_B(0,0) = \xi_\infty $, we may compute
    \begin{equation} \label{Psi-deriv}
        D\Psi(0,0) = \xi_\infty I,
    \end{equation}
    where $I$ is the identity matrix. Hence, $\Psi$ is a local diffeomorphism near $(0,0)$, and the local coordinate chart described by coordinates $(\eta_A,\eta_B)$ gives rise to another local coordinate chart described by coordinates $(h_A,h_B)$.
\end{proof}

\begin{proposition} \label{prop:M-generating}
    There is some analytic generating function $M(\eta_A,\eta_B)$ defined in some neighborhood of $(0,0)$ such that $ M(0,0) = 0 $ and the gluing map $ G: (\eta_A,\xi_A) \mapsto (\xi_B,\eta_B) $ is given implicitly by
    \begin{equation} \label{M-def}
        \xi_A = \pderiv{M}{\eta_A} \textand \xi_B = \pderiv{M}{\eta_B}.
    \end{equation}
    Furthermore, $M$ is symmetric in the sense that $ M(\eta_A,\eta_B) = M(\eta_B,\eta_A) $.
\end{proposition}

\begin{proof}
    This is a consequence of Lemma \ref{lem:coord} and the fact that $ G \circ I: (\xi_A,\eta_A) \mapsto (\xi_B,\eta_B) $ is area-preserving and orientation-reversing.
    Symmetry of $M$ follows from the fact that the gluing map can also be realized as $ G: (\eta_B,\xi_B) \mapsto (\xi_A,\eta_A) $, and thus interchanging the subscripts $A$ and $B$ produces the same function $M$.
\end{proof}

The following is an immediate consequence:

\begin{corollary} \label{cor:q-function}
    Define the analytic function $ v := \pderiv{M}{\eta_A} \circ \Psi^{-1} $. Then,
    \begin{equation} \label{q-function}
        \xi_A = v(h_A,h_B) \textand \xi_B = v(h_B,h_A)
    \end{equation}
    for $(h_A,h_B)$ near $(0,0)$.
\end{corollary}

Note that the normal form $N$ and gluing map $G$ are uniquely determined by the system and two systems with the same normal form and gluing map are analytically conjugate to each other near the homoclinic orbit, so to prove Theorem \ref{th:conj}, it suffices to prove that two systems with the same length spectrum necessarily have the same normal form and gluing map.

\section{Asymptotic Triangular Power Series} \label{sec:series}

We are interested in asymptotic power series expressions of the form
\begin{equation} \label{asym-series}
    \sum_{p,q,i,j} c_{pq}^{ij} m^i n^j \parens{\lambda^{2m}}^p \parens{\lambda^{2n}}^q.
\end{equation}
In this section, we show that the Birkhoff energies $h_A$ and $h_B$ have this type of expansion. To do so, we write $ z_A := \xi_\infty^2 \lambda^{2m} $ and $ z_B := \xi_\infty^2 \lambda^{2n} $ and work with formal power series expressions of the form
\begin{equation} \label{formal-series}
    F(m,n,z_A,z_B) = \sum_{p,q,i,j} a_{pq}^{ij} m^i n^j z_A^p z_B^q.
\end{equation}
Power series similar to ours are used in \cite{DKL,KLZ}, but in two variables instead of four.

For such a power series $F$, we adopt the notation $ \coef{F}{pq}{ij} = a_{pq}^{ij} $ for the coefficient of $ m^i n^j z_A^p z_B^q $. We also define
$$ \coefp{F}{pq} = \sum_{i,j} a_{pq}^{ij} m^i n^j. $$
For any polynomial $P$ in the variables $m$ and $n$, we write $ \deg P = (r,s) $ if $P$ has degree $r$ in $m$ and degree $s$ in $n$. By $ \deg P \leq (r,s) $, we mean $P$ has degree less than or equal to $r$ in $m$ and degree less than or equal to $s$ in $n$.

\begin{definition} \label{def:tri}
    We say a formal power series $F$ is \keyterm{triangular} if for all $p$ and $q$, $\coefp{F}{pq}$ is a polynomial of finite degree in $m$ and $n$ with $ \deg \coefp{F}{pq} \leq (p,q) $, or equivalently, if $F$ can be written in the form
    \begin{equation} \label{triangular-def}
        F = \sum_{p=0}^\infty \sum_{q=0}^\infty \sum_{i=0}^p \sum_{j=0}^q a_{pq}^{ij} m^i n^j z_A^p z_B^q.
    \end{equation}
    Furthermore, if $ \deg \coefp{F}{pq} \leq \parens{\max\set{0,p-1}, q} $ (respectively $ \deg \coefp{F}{pq} \leq \parens{p, \max\set{0,q-1}} $) for all $p$ and $q$, then $F$ is \keyterm{strictly triangular in $z_A$} (respectively \keyterm{strictly triangular in $z_B$}).
    A formal power series will be called \keyterm{strictly triangular (in both)} if it is strictly triangular in both $z_A$ and $z_B$, or equivalently, if $ \deg \coefp{F}{pq} \leq \parens{\max\set{0,p-1}, \max\set{0,q-1}} $.
    
    For any $ \nu \in \N $, a formal power series $F$ is \keyterm{triangular up to degree $\nu$} if for all $p$ and $q$ with $ p + q \leq \nu $, $ \coefp{F}{pq} \leq (p,q) $.
    Similarly, $F$ is \keyterm{strictly triangular} in $z_A$ or $z_B$ \keyterm{up to degree $\nu$} if the respective condition $ \deg \coefp{F}{pq} \leq \parens{\max\set{0,p-1}, q} $ or $ \deg \coefp{F}{pq} \leq \parens{p, \max\set{0,q-1}} $ holds whenever $ p+q \leq \nu $, and \keyterm{strictly triangular (in both) up to degree $\nu$} if $ \deg \coefp{F}{pq} \leq \parens{\max\set{0,p-1}, \max\set{0,q-1}} $.
\end{definition}

In general, a triangular power series that is strictly triangular in $z_A$ takes the form
\begin{equation} \label{strict-tri-A}
    \sum_{p,q} \sum_{i=0}^{\max\set{0,p-1}} \sum_{j=0}^q a_{pq}^{ij} m^i n^j z_A^p z_B^q,
\end{equation}
one that is strictly triangular in $z_B$ takes the form
\begin{equation} \label{strict-tri-B}
    \sum_{p,q} \sum_{i=0}^p \sum_{j=0}^{\max\set{0,q-1}} a_{pq}^{ij} m^i n^j z_A^p z_B^q,
\end{equation}
and one that is strictly triangular in both takes the form
\begin{equation} \label{strict-tri}
    \sum_{p,q} \sum_{i=0}^{\max\set{0,p-1}} \sum_{j=0}^{\max\set{0,q-1}} a_{pq}^{ij} m^i n^j z_A^p z_B^q.
\end{equation}
A power series that is triangular up to degree $\nu$ takes the form
\begin{equation} \label{tri-norm}
    \sum_{p+q \leq \nu} \sum_{i=0}^p \sum_{j=0}^q a_{pq}^{ij} m^i n^j z_A^p z_B^q + R_{\nu+1},
\end{equation}
where the remainder term $R_{\nu+1}$ is a power series whose terms all have degree at least $\nu+1$ in the variables $z_A$ and $z_B$. Similar expressions with different bounds for $i$ and $j$ are valid for power series that are strictly triangular in $z_A$, $z_B$, or both up to degree $\nu$.

In section \ref{sec:recover}, we will often find it useful to express a triangular power series in the form
\begin{equation} \label{tri-series-alt}
    F = \sum_{p=0}^\infty \sum_{q=0}^\infty \sum_{k=0}^\infty \sum_{\ell=0}^\infty a_{p+k,q+\ell}^{k\ell} m^k n^\ell z_A^{p+k} z_B^{q+\ell}
\end{equation}
rather than in the form (\ref{triangular-def}).

The following are properties of triangular power series:
\begin{enumerate}[label = (T\arabic*)]
    \item A linear combination of power series that are triangular (respectively, strictly triangular in $z_A$, $z_B$, or both) up to degree $\nu$ is triangular (respectively, strictly triangular in $z_A$, $z_B$, or both) up to degree $\nu$.
    \item A product of power series that are triangular (respectively, strictly triangular in $z_A$, $z_B$, or both) up to degree $\nu$ is triangular (respectively, strictly triangular in $z_A$, $z_B$, or both) up to degree $\nu$.
    \item An analytic function (applied as a composition of formal power series) of a power series that is triangular (respectively, strictly triangular in $z_A$, $z_B$, or both) up to degree $\nu$ is triangular (respectively, strictly triangular in $z_A$, $z_B$, or both) up to degree $\nu$.
    \item If $F$ is strictly triangular in $z_A$ up to degree $\nu$ and is a multiple of $z_A$, then $\frac{F}{z_A}$ is triangular up to degree $\nu-1$. Furthermore, if $F$ is strictly triangular in $z_B$ up to degree $\nu$, then so is $\frac{F}{z_A}$ up to degree $\nu-1$.
\end{enumerate}
Properties (T1), (T2), and (T3) can be proven directly from the definitions. Property (T4) follows from the fact that a power series $F$ that is strictly triangular in $z_A$ up to degree $\nu$ and is a multiple of $z_A$ can be written in the form
$$ F = \sum_{p=1}^\nu \sum_{q=0}^{\nu-p} \sum_{i=0}^{p-1} \sum_{j=0}^q a_{pq}^{ij} m^i n^j z_A^p z_B^q + R_{\nu+1}. $$
Note that all of the above statements remain valid if all of the restrictions ``up to degree $\nu$'' and ``up to degree $\nu-1$'' are removed.

We devote the remainder of this section to proving the following:

\begin{proposition} \label{prop:h-triangular}
    The Birkhoff energy $h_A$ can be expressed as a strictly triangular formal power series that is a multiple of $z_A$,
    \begin{equation} \label{hA-triangular}
        h_A = \sum_{p=1}^\infty \sum_{q=0}^\infty \sum_{i=0}^{p-1} \sum_{j=0}^{\max\set{0,q-1}} \hc{pq}{ij} m^i n^j z_A^p z_B^q,
    \end{equation}
    where $ z_A = \xi_\infty^2 \lambda^{2m} $ and $ z_B = \xi_\infty^2 \lambda^{2n} $.
    Furthermore, $h_B$ can also be expressed as a strictly triangular formal power series that is a multiple of $z_B$ with
    \begin{equation} \label{hB-triangular}
        h_B = \sum_{p=0}^\infty \sum_{q=1}^\infty \sum_{i=0}^{\max\set{0,p-1}} \sum_{j=0}^{q-1} \hc{qp}{ji} m^i n^j z_A^p z_B^q.
    \end{equation}
\end{proposition}

Note that the coefficients in the power series for $h_A$ and $h_B$ are the same, but with the variable $m$ swapped with $n$ and the variable $z_A$ swapped with $z_B$.

\begin{remark}
    Although $h_A$ and $h_B$ are analytic functions of $m$, $n$, $z_A$, and $z_B$, this does not necessarily imply that the power series converge for large $m$ and $n$ after substituting $ z_A = \xi_\infty^2 \lambda^{2m} $ and $ z_B = \xi_\infty^2 \lambda^{2n} $. We do not address the problem of convergence, but we do prove the asymptotic approximations
    \begin{equation} \label{r-s-approx-h}
        h_A = \sum_{p=1}^r \sum_{q=0}^s \sum_{i,j} \hc{pq}{ij} m^i n^j z_A^p z_B^q + o(z_A^r z_B^s)
    \end{equation}
    for large $m$ and $n$ and similar approximations for $h_B$.
    These will be sufficient for our purposes.
\end{remark}

\begin{lemma} \label{lem:h-u-implicit}
    Define the analytic function
    \begin{equation} \label{u-def}
        u(h_A,h_B) := \frac{v(h_A,h_B)^2}{\xi_\infty^2} - 1,
    \end{equation}
    where $v$ is as defined in Corollary \ref{cor:q-function}.
    Then, the Birkhoff energies satisfy
    \begin{equation} \label{hA-implicit}
        h_A = \parens{1 + u(h_A,h_B)} \parens{1 + \delta(h_A)}^{2m} z_A
    \end{equation}
    and
    \begin{equation} \label{hB-implicit}
        h_B = \parens{1 + u(h_B,h_A)} \parens{1 + \delta(h_B)}^{2n} z_B.
    \end{equation}
\end{lemma}

\begin{proof}
    By (\ref{orbit-symmetry}), $ \eta_A = \xi_A \mu(h_A)^{2m} $, so we have
    $$ h_A = \xi_A \eta_A = \xi_A^2 \mu(h_A)^{2m} = v(h_A,h_B)^2 \parens{1 + \delta(h_A)}^{2m} \lambda^{2m}, $$
    from which (\ref{hA-implicit}) follows. The identity (\ref{hB-implicit}) can be proven similarly.
\end{proof}

The problem of solving the equations (\ref{hA-implicit}) and (\ref{hB-implicit}) for the power series expansions of $h_A$ and $h_B$ in terms of $z_A$ and $z_B$ is a form of multivariable Lagrange inversion.
Explicit formulas for the power series in multivariable Lagrange inversion are known (see, for example, \cite{Gessel}). However, these formulas are complicated and we do not need to completely recover these power series, so we instead work with the equations (\ref{hA-implicit}) and (\ref{hB-implicit}) directly through an inductive process. Specifically, we substitute an order $\nu$ approximation of $h_A$ and $h_B$ into (\ref{hA-implicit}) to produce an order $\nu+1$ approximation of $h_A$.

\begin{proof}[Proof of Proposition \ref{prop:h-triangular}]
    It will be proven by induction on $\nu$ that $h_A$ and $h_B$ are strictly triangular up to degree $\nu$, that is,
    \begin{equation} \label{h-triangular-pf-1}
        h_A = \sum_{p=1}^\nu \sum_{q=0}^{p-\nu} \sum_{i=0}^{p-1} \sum_{j=0}^{\max\set{0,q-1}} \hc{pq}{ij} m^i n^j z_A^p z_B^q + o\parens{\norm{(z_A,z_B)}^\nu}
    \end{equation}
    and
    \begin{equation} \label{h-triangular-pf-2}
        h_B = \sum_{p=1}^\nu \sum_{q=0}^{p-\nu} \sum_{i=0}^{p-1} \sum_{j=0}^{\max\set{0,q-1}} \hc{pq}{ij} n^i m^j z_B^p z_A^q + o\parens{\norm{(z_A,z_B)}^\nu}
    \end{equation}
    where the $o$ notation describes the behavior as $ m,n \rightarrow \infty $.
    
    First notice that since $ h_A \rightarrow 0 $ and $ h_B \rightarrow 0 $ as $ m,n \rightarrow \infty $, we have from (\ref{hA-implicit}) the first order approximation
    $$ h_A = z_A + o\parens{\norm{(z_A,z_B)}}. $$
    This establishes (\ref{h-triangular-pf-1}) for $\nu=1$.
    The statement (\ref{h-triangular-pf-2}) is similarly true for $\nu=1$.
    
    Now assume that (\ref{h-triangular-pf-1}) and (\ref{h-triangular-pf-2}) are true for some $ \nu \geq 1 $. To prove that the statement holds for $\nu+1$, it suffices to show that for all nonnegative integers $p$ and $q$ with $ p+q = \nu $, $ \deg \coefp{h_A}{p+1,q} \leq \parens{p, \max\set{0,q-1}} $. The exclusion of the coefficient $\coefp{h_A}{0,\nu+1}$ is valid since the expression (\ref{hA-implicit}) implies $h_A$ must be a multiple of $z_A$.
    
    Let $p$ and $q$ be nonnegative integers such that $ p + q = \nu $. By (\ref{hA-implicit}), we have
    \begin{equation} \label{h-triangular-pf-3}
        \coefp{h_A}{p+1,q} = \sum_{r=0}^p \sum_{s=0}^q \coefp{1 + u(h_A,h_B)}{rs} \coefp{\parens{1 + \delta(h_A)}^{2m}}{p-r,q-s}.
    \end{equation}
    Suppose that $ 0 \leq r \leq p $ and $ 0 \leq s \leq q $. If $ (r,s) \neq (0,0) $, the coefficient of $ 1 + u(h_A,h_B) $ can be written as the finite sum
    $$ \coefp{1 + u(h_A,h_B)}{rs} = \sum_{i+j \leq r+s} u_{ij} \coefp{h_A^i h_B^j}{rs}, $$
    where
    $$ u(h_A,h_B) = \sum_{i,j} u_{ij} h_A^i h_B^j $$
    is the power series expansion of $u$.
    By the inductive hypothesis, $h_A$ and $h_B$ are both strictly triangular up to degree $\nu$, so by property (T2), each $ h_A^i h_B^j $ is also strictly triangular up to degree $\nu$ and since $ r+s \leq \nu $,
    $$ \deg \coefp{h_A^i h_B^j}{rs} \leq \parens{\max\set{0,r-1}, \max\set{0,s-1}}, $$
    and as a linear combination of these polynomials,
    \begin{equation} \label{h-triangular-pf-4}
        \deg \coefp{1 + u(h_A,h_B)}{rs} \leq \parens{\max\set{0,r-1}, \max\set{0,s-1}}.
    \end{equation}
    Since $ \coefp{1 + u(h_A,h_B)}{0,0} = 1 $, (\ref{h-triangular-pf-4}) clearly also holds when $ (r,s) = (0,0) $.
    
    We now deal with the other factor in (\ref{h-triangular-pf-3}). We first expand
    $$ \coefp{\parens{1 + \delta(h_A)}^{2m}}{rs} = \sum_{i=1}^r \binom{2m}{i} \coefp{\delta(h_A)^i}{rs} = \sum_{i=0}^r \binom{2m}{i} \coefp{\parens{\frac{\delta(h_A)}{z_A}}^i}{r-i,s}. $$
    Here the upper bound of $r$ is due to the fact that $\delta(h_A)^i$ is a multiple of $z_A^i$, and thus its coefficient in front of $ z_A^r z_A^s $ must be zero whenever $ i > r $.
    By property (T3) and the inductive hypothesis, $\delta(h_A)$ is strictly triangular up to degree $\nu$. Also, $\delta(h_A)$ is a multiple of $z_A$, so by property (T4), $\frac{\delta(h_A)}{z_A}$ is triangular in $z_A$ and strictly triangular in $z_B$ up to degree $\nu-1$. By property (T2), so is its power $ \parens{\frac{\delta(h_A)}{z_A}}^i $, which implies $ \deg \coefp{\parens{\frac{\delta(h_A)}{z_A}}^i}{r-i,s} \leq \parens{r-i, \max\set{0,s-1}} $. Also, $ \deg \binom{2m}{i} = (i,0) $, so for all $i$,
    $$ \deg \binom{2m}{i} \coefp{\delta(h_A)^i}{rs} = \deg \binom{2m}{r} + \deg \coefp{\parens{\frac{\delta(h_A)}{z_A}}^i}{r-i,s} \leq \parens{r, \max\set{0,s-1}}. $$
    Hence,
    \begin{equation} \label{h-triangular-pf-5}
        \deg \coefp{\parens{1 + \delta(h_A)}^{2m}}{rs} \leq \parens{r, \max\set{0,s-1}}.
    \end{equation}
    It can be easily verified that this also holds for $ r = 0 $, since in this case $ \coefp{\parens{1 + \delta(h_A)}^{2m}}{rs} $ is equal to $1$ when $ s = 0 $ and $0$ when $ s \geq 1 $.
    
    Applying (\ref{h-triangular-pf-4}) and (\ref{h-triangular-pf-5}) to (\ref{h-triangular-pf-3}) yields $ \deg \coefp{h_A}{p+1,q} \leq \parens{p, \max\set{0,q-1}} $, completing the proof of the inductive step for $h_A$.
    The fact that $ \deg \coefp{h_B}{p,q+1} \leq \parens{\max\set{0,p-1}, q} $ follows by replacing $m$ with $n$ and $z_A$ with $z_B$ in the above proof.
\end{proof}

\section{Computing Perimeters of Periodic Orbits} \label{sec:length}

In this section, we derive a formula for the length spectral data in terms of the dynamics in Birkhoff coordinates by considering the perimeter of a cyclicity-$2$ periodic orbit as a symplectic action and computing this action.

For $ i,j \in \set{1,2,3} $, we define $ L_{ij}(s,s') $ to be the distance between the point $\gamma_i(s)$ on scatterer $i$ and $\gamma_j(s')$ on scatterer $j$.
The following is well-known (see, for example, Section 9.2 of \cite{KH}):

\begin{lemma} \label{lem:L-generating}
    The function $L_{ij}$ is a generating function in the sense that for $ (s,r) \in \cspace_i $ and $ (s',r') \in \cspace_j $, $ (s',r') = \cmap(s,r) $ if and only if $ r = -\pderiv{L}{s}(s,s') $ and $ r' = \pderiv{L}{s'}(s,s') $.
\end{lemma}

The next three lemmas and Proposition \ref{prop:length-action} are very similar to results that appear in \cite{KLZ}.

\begin{lemma} \label{lem:Hamiltonian}
    Define a Hamiltonian $ H: U_1 \cup U_2 \rightarrow \R $ by
    \begin{equation} \label{eq:Hamiltoniandef}
        H(\xi,\eta) = \int_0^{\xi\eta} \log \mu(\tau) d\tau.
    \end{equation}
    Then, the Hamiltonian flow of $H$ is exactly $N^t$, defined by
    \begin{equation} \label{fractional-N}
        N^t(\xi,\eta) = \parens{\mu(\xi\eta)^t \xi, \frac{\eta}{\mu(\xi\eta)^t}}.
    \end{equation}.
\end{lemma}

\begin{proof}
    Let $ (\xi,\eta) = \parens{\mu(\xi_0 \eta_0)^t \xi_0, \frac{\eta_0}{\mu(\xi_0 \eta_0)^t}} $ be an orbit of $N^t$. It can be verified that through direct computation that $ \deriv{\xi}{t} = \pderiv{H}{\eta} $ and $ \deriv{\eta}{t} = -\pderiv{H}{\xi} $.
\end{proof}

\begin{lemma} \label{lem:symplecticA}
    The \keyterm{symplectic action} of $H$ on an integral curve $ \zeta: [k_1, k_2] \rightarrow \R^2 $ of the Hamiltonian flow is given by
    \begin{equation} \label{symplectic-action}
        A(\zeta) := \int_\zeta (\eta\dd\xi - H \dd t) = (k_2 - k_1) \Sigma(h),
    \end{equation}
    where $h$ is the Birkhoff energy on the curve $\zeta$ and
    \begin{equation} \label{Sigma-def}
        \Sigma(h) := h \log \mu(h) - \int_0^h \log \mu(\tau) \dd\tau.
    \end{equation}
\end{lemma}

The $1$-form $ \eta\dd\xi - H \dd t $ is called the \keyterm{Poincar\'{e}-Cartan integral invariant} and has many applications in the theory of Hamiltonian dynamics. For more on this topic, see \cite{Arnold}.

\begin{proof}[Proof of Lemma \ref{lem:symplecticA}]
    Define $ \hat{H}(h) := H(\xi,\eta) $ for $ h = \xi\eta $.
    Suppose that $ \zeta(0) = (\xi_0,\eta_0) $ so that $ \zeta(t) = \parens{\mu(h)^t \xi_0, \frac{\eta_0}{\mu(h)^t}} $, where $ h = \xi_0 \eta_0 $. Then,
    $$ A(\zeta) = \int_\zeta (\eta\dd\xi - H \dd t) = \int_{k_1}^{k_2} \parens{h \log\mu(h) - \hat{H}(h)} dt = (k_2 - k_1) \parens{h \hat{H}'(h) - \hat{H}(h)} = (k_2 - k_1) \Sigma(h). $$
\end{proof}

\begin{lemma} \label{lem:length-expr}
    There are analytic functions $ S_1: U_1 \rightarrow \R $ and $ S_2: U_2 \rightarrow \R $ such that for $ \set{i,j} = \set{1,2} $ and $ (\xi,\eta) \in U_i \cup N^{-1}(U_j) $,
    \begin{equation} \label{L-action}
        L_{ij}(s,s') = \ell_0 + \Sigma(\xi\eta) + S_j\parens{N(\xi,\eta)} - S_i(\xi,\eta),
    \end{equation}
    where $ (s,r) = \Phi_i(\xi,\eta) $ and $ (s',r') = \Phi_j\parens{N(\xi,\eta)} $.
    Furthermore, for $ i \in \set{1,2} $,
    \begin{equation} \label{S-involution}
        S_i(\xi,\eta) + S_i(\eta,\xi) = -\xi\eta.
    \end{equation}
    Also, there is an analytic function $ S_-: U_- \rightarrow \R $ such that for $ (\xi,\eta) \in U_- $,
    \begin{equation} \label{L-action-neg}
        L_{31}(s,s') = \ell_0 + \Sigma(\xi\eta) + S_1\parens{N(\xi,\eta)} - S_-(\xi,\eta),
    \end{equation}
    where $ (s,r) = \Phi_-(\xi,\eta) $ and $ (s',r') = \Phi_1\parens{N(\xi,\eta)} $, and an analytic function $ S_+: U_+ \rightarrow \R $ such that for $ (\xi,\eta) \in U_+ $,
    \begin{equation} \label{L-action-pos}
        L_{13}(s,s') = \ell_0 + \Sigma(\xi\eta) + S_+(\xi,\eta) - S_1\parens{N^{-1}(\xi,\eta)},
    \end{equation}
    where $ (s,r) = \Phi_1^{-1}\parens{N^{-1}(\xi,\eta)} $ and $ (s',r') = \Phi_+(\xi,\eta) $.
    Furthermore, for $ (\xi,\eta) \in U_- $,
    \begin{equation} \label{S-involution-3}
        S_-(\xi,\eta) + S_+(\eta,\xi) = -\xi\eta.
    \end{equation}
\end{lemma}

\begin{proof}
    The functions $\Phi_1$ and $\Phi_2$ are area-preserving and orientation-preserving transformations in two dimensions, and are thus symplectic. Therefore, for $ i \in \set{1,2} $, there are functions $S_i(\xi,\eta)$ such that
    $$ \Phi_i^* \, r \dd s = \eta\dd\xi + \dd S $$
    (see, for example, Section 45 of \cite{Arnold}).
    For $ \set{i,j} = \set{1,2} $, if $(s,r)$, $(\xi,\eta)$, and $(s',r')$, are as in the lemma and $ (\xi',\eta') = N(\xi,\eta) $, then
    $$ (\Phi_i \times \Phi_j)^*(r' \dd s' - r \dd s) = \eta' \dd\xi' - \eta\dd\xi + \dd S_j(\xi',\eta') - \dd S_i(\xi,\eta). $$
    If $\zeta$ is the integral curve of $H$ from $(\xi,\eta)$ to $(\xi',\eta')$, then
    $$ (\Phi_i \times \Phi_j)^*(r' \dd s' - r \dd s) = dA(\zeta) + \dd S_j(\xi',\eta') - \dd S_i(\xi,\eta). $$
    If $ \ell = L_{ij}(s,s') $, then by Lemma \ref{lem:L-generating},
    $$ L_{ij}^* \dd\ell = \pderiv{L_{ij}}{s} \dd s + \pderiv{L_{ij}}{s'} \dd s' = r' \dd s' - r \dd s. $$
    By integration, using the value of $A(\zeta)$ in Lemma \ref{lem:symplecticA},
    $$ L_{ij}(s,s') = \Sigma(\xi\eta) + S_j(\xi',\eta') - S_i(\xi,\eta) + C $$
    for some constant $C$. If $s$ and $s'$ are the $2$-periodic points, then $ (\xi,\eta) = (\xi',\eta') = (0,0) $, $ A(\zeta) = 0 $, and $ L_{ij}(s,s') = \ell_0 $, so $ C = \ell_0 $ and we have (\ref{L-action}).
    
    By statements (d) and (e) of Proposition \ref{prop:BNF}, if $ i \in \set{1,2} $, then
    \begin{align*}
        \dd S_i(\eta,\xi) & = I^* \, \dd S_i(\xi,\eta) = I^* \, \Phi_i^* \, r \dd s - I^* \, \eta\dd\xi = \Phi_i^* \, I_0^* \, (r \dd s) - \xi\dd\eta = -\Phi_i^* \, (r \dd s) - \xi\dd\eta \\
        & = -\dd S_i(\xi,\eta) - \eta\dd\xi - \xi\dd\eta = -\dd S_i(\xi,\eta) - d(\xi\eta).
    \end{align*}
    Hence, $ S_i(\xi,\eta) + S_i(\eta,\xi) + \xi\eta = C_i $ for some constant $C_i$. By setting $ (\xi,\eta) = (0,0) $, $i=1$, and $j=2$ in (\ref{L-action}), it is clear that $ C_1 = C_2 $. Since (\ref{L-action}) remains unchanged by adding the same arbitrary constant to $S_1$ and $S_2$, we may set $ C_1 = C_2 = 0 $.
    This proves (\ref{S-involution}).
    
    To satisfy the condition (\ref{L-action-neg}), we may simply define, for $ (\xi,\eta) \in U_- $,
    $$ S_-(\xi,\eta) := \ell_0 + \Sigma(\xi\eta) + S_1\parens{N(\xi,\eta)} - L_{31}(s,s'), $$
    where $ (s,r) = \Phi_-(\xi,\eta) $ and $ (s',r') = \Phi_1\parens{N(\xi,\eta)} $. Since the involution preserves the points of collision $\gamma_i(s)$, $L_{31}(s,s')$ is also the distance between the points at which the collisions $\Phi_+(\eta,\xi)$ and $\Phi_1\parens{N^{-1}(\eta,\xi)}$ occur. Therefore, to satisfy the condition (\ref{L-action-pos}), we may define
    $$ S_+(\eta,\xi) := L_{31}(s,s') - \ell_0 - \Sigma(\xi\eta) + S_1\parens{N^{-1}(\eta,\xi)}. $$
    We then have
    $$ S_-(\xi,\eta) + S_+(\eta,\xi) = S_1\parens{N(\xi,\eta)} + S_1\parens{N^{-1}(\eta,\xi)}. $$
    By (\ref{involution-Birkhoff}), $ N^{-1}(\eta,\xi) = I\parens{N(\xi,\eta)} $, so (\ref{S-involution-3}) follows directly from (\ref{S-involution}).
\end{proof}

\begin{proposition} \label{prop:length-action}
    For all sufficiently large $m$ and $n$, the perimeter of the orbit $\set{x_k^{m,n}}_{k=0}^{2m+2n-1}$ is
    \begin{equation} \label{length-action}
        \ell_{m,n} = (2m+2n)\ell_0 + 2L_\infty + 2m \Sigma(h_A) + 2n \Sigma(h_B) + h_A + h_B - 2\tilde{M}(h_A,h_B)
    \end{equation}
    for some constant $L_\infty$, where $\ell_0$ is the distance between the $2$-periodic points on scatterers $1$ and $2$, $ \tilde{M} := M \circ \Psi^{-1} $, and $\Psi$ is as defined in (\ref{Psi-p-h}).
\end{proposition}

\begin{proof}
    By repeated application of Lemma \ref{lem:length-expr}, the length of the orbit between the collision points $x_0^{m,n}$ and $x_{2m}^{m,n}$ is $ 2m \ell_0 + 2m \Sigma(h_A) + S_+(\eta_A,\xi_A) - S_-(\xi_A,\eta_A) $, and the length of the orbit between the collision points $x_{2m}^{m,n}$ and $x_{2m+2n}^{m,n}$ is $ 2n \ell_0 + 2n \Sigma(h_B) + S_+(\eta_B,\xi_B) - S_-(\xi_B,\eta_B) $. Therefore,
    $$ \ell_{m,n} = (2m+2n) \ell_0 + 2m \Sigma(h_A) + 2n \Sigma(h_B) + S_+(\eta_A,\xi_A) - S_-(\xi_A,\eta_A) + S_+(\eta_B,\xi_B) - S_-(\xi_B,\eta_B). $$
    By applying (\ref{S-involution-3}), we have
    $$ \ell_{m,n} = (2m+2n) \ell_0 + 2m \Sigma(h_A) + 2n \Sigma(h_B) - h_A - h_B - 2 S_-(\xi_A,\eta_A) - 2 S_-(\xi_B,\eta_B). $$
    Notice that $ \tilde{M}(h_A,h_B) = M(\eta_A,\eta_B) $. It suffices to prove that
    $$ M(\eta_A,\eta_B) = S_-(\xi_A,\eta_A) + S_-(\xi_B,\eta_B) + h_A + h_B - L_\infty $$
    for some constant $L_\infty$.
    
    Consider the commutative diagram in Figure \ref{fig:length-action-pf}. Notice that
    $$ (I_0 \circ \Phi_-)^* \, r \dd s = \Phi_-^* \, I_0^* \, r \dd s = -\Phi_-^* \, r \dd s = -\eta_A \dd \xi_A - \dd S_-(\xi_A,\eta_A) $$
    and
    $$ (\Phi_- \circ G \circ I)^* \, r \dd s = (G \circ I)^* \parens{\eta_B \dd \xi_B + \dd S(\xi_B,\eta_B)}. $$
    Equating these differential forms yields
    $$ \eta_A \dd \xi_A + \eta_B \dd \xi_B + \dd S_-(\xi_A,\eta_A) + \dd S_-(\xi_B,\eta_B) = 0, $$
    or equivalently,
    $$ \dd S_-(\xi_A,\eta_A) + \dd S_-(\xi_B,\eta_B) + \dd h_A + \dd h_B = \xi_A \dd \eta_A + \xi_B \dd \eta_B. $$
    The equations (\ref{M-def}) imply that $ \xi_A \dd \eta_A + \xi_B \dd \eta_B = \dd M(\eta_A,\eta_B) $, so integrating completes the proof.
\end{proof}

\begin{figure}[ht]
    \centering
    \begin{tikzpicture}
        \node (A) at (0,2) {$(\xi_A,\eta_A)$};
        \node (B) at (3,2) {$(\xi_B,\eta_B)$};
        \node (C) at (0,0) {$(s,r)$};
        \node (D) at (3,0) {$(s,-r)$};
        
        \path[|->] (A) edge node[above]{$ G \circ I $} (B);
        \path[|->] (A) edge node[left]{$\Phi_-$} (C);
        \path[|->] (C) edge node[above]{$I_0$} (D);
        \path[|->] (B) edge node[right]{$\Phi_-$} (D);
    \end{tikzpicture}
    \caption{Commutative diagram used in the proof of Proposition \ref{prop:length-action}.}
    \label{fig:length-action-pf}
\end{figure}

\begin{corollary} \label{cor:l-UT}
    The length spectral data can be expressed as a strictly triangular power series
    \begin{equation} \label{length-triangular}
        \ell_{m,n} = (2m + 2n) \ell_0 + \sum_{p=0}^\infty \sum_{q=0}^\infty \sum_{i=0}^{\max\set{0,p-1}} \sum_{j=0}^{\max\set{0,q-1}} \lc{pq}{ij} m^i n^j z_A^i z_B^j,
    \end{equation}
    where $ z_A = \xi_\infty^2 \lambda^{2m} $ and $ z_B = \xi_\infty^2 \lambda^{2n} $.
\end{corollary}

\begin{proof}
    This follows from Proposition \ref{prop:length-action}, the fact that $\Sigma$ and $\tilde{M}$ are analytic, Proposition \ref{prop:h-triangular}, and properties (T3) and (T1) of triangular power series.
\end{proof}

Just as in Proposition \ref{prop:h-triangular}, we do not necessarily have convergence for large $m$ and $n$ after substituting $ z_A = \xi_\infty^2 \lambda^{2m} $ and $ z_B = \xi_\infty^2 \lambda^{2n} $, but we do have the asymptotic approximations
\begin{equation} \label{r-s-approx-l}
    \ell_{m,n} = (2m+2n) \ell_0 + \sum_{p=0}^r \sum_{q=0}^s \sum_{i,j} \lc{pq}{ij} m^i n^j z_A^p z_B^q + o(z_A^r z_B^s)
\end{equation}
for large $m$ and $n$.

\begin{remark} \label{rem:parity}
    As remarked in the introduction, the parity of $m$ and $n$ determine which scatterers the perpendicular bounces occur in the cyclicity-$2$ orbits.
    However, the Birkhoff coordinates do not distinguish between scatterers $1$ and $2$, and we have computed asymptotics of the length spectrum entirely from the dynamics of the Birkhoff coordinates, so this distinction is not realized in the length spectral data.
    This is a source of difficulty in removing the requirement of reflexive symmetry from length spectrum rigidity.
\end{remark}

We finish this section by determining explicitly the first-order approximation of the length spectral data.

\begin{proposition} \label{prop:l-order-1}
    It holds that $ \lc{00}{00} = 2L_\infty $ and $ \lc{10}{00} = \lc{01}{00} = -1 $.
\end{proposition}

\begin{proof}
    It is clear from (\ref{Sigma-def}) that $ \Sigma(h) = \Oo(h^2) $. Therefore, $ \Sigma(h_A) = \Oo(z_A^2) $ and $ \Sigma(h_B) = \Oo(z_B^2) $, so by the formula (\ref{length-action}), the coefficients $\lc{00}{00}$, $\lc{10}{00}$, and $\lc{01}{00}$ are determined by the first-order approximation of $M$ at $(0,0)$ in the sense that
    $$ \lc{00}{00} + \lc{10}{00} z_A + \lc{01}{00} z_B = 2L_\infty + h_A + h_B - 2\tilde{M}(h_A,h_B) + o(z_A + z_B). $$
    Since $ h_A \sim z_A $ and $ h_B \sim z_B $,
    $$ \lc{00}{00} + \lc{10}{00} h_A + \lc{01}{00} h_B = 2L_\infty + h_A + h_B - 2\tilde{M}(h_A,h_B) + o(h_A + h_B). $$
    
    The equations (\ref{M-def}) uniquely determine the first-order approximation of $M$, specifically,
    $$ M(\eta_A,\eta_B) = \xi_\infty \eta_A + \xi_\infty \eta_B + o(\eta_A + \eta_B). $$
    Since $ \xi_A \sim \xi_\infty $, $ \xi_B \sim \xi_\infty $, $ h_A \sim z_A $, and $ h_B \sim z_B $ for large $m$ and $n$, this simplifies to
    \begin{equation} \label{M-order-1}
        \tilde{M}(h_A,h_B) = h_A + h_B + o(h_A + h_B).
    \end{equation}
    Therefore,
    $$ \lc{00}{00} + \lc{10}{00} h_A + \lc{01}{00} h_B = 2L_\infty - h_A - h_B + o(h_A + h_B), $$
    and the proof is complete.
\end{proof}

\section{Recovering the Dynamics from the Length Spectral Data} \label{sec:recover}

The goal of this section is to show that the functions $\mu$ and $\tilde{M}$ can be recovered from the length spectral data $\ell_{m,n}$ via the power series expansion (\ref{length-triangular}).
We write the power series of the function $\mu$ as
\begin{equation} \label{mu-series}
    \mu(h) = \lambda \, \parens{1 + \sum_{j=1}^\infty \delta_j h^j},
\end{equation}
the power series of $M$ as
\begin{equation} \label{M-series-p}
    M(\eta_A,\eta_B) = \sum_{i,j} \hat{a}_{ij} \eta_A^i \eta_B^j,
\end{equation}
and the power series of $\tilde{M}$ as
\begin{equation} \label{M-series-h}
    \tilde{M}(h_A,h_B) = \sum_{i,j} a_{ij} h_A^i h_B^j.
\end{equation}
Note that by symmetry of $M$, $ \hat{a}_{ij} = \hat{a}_{ji} $ and $ a_{ij} = a_{ji} $.

In our proof, we will consider the coefficients $\hc{pq}{ij}$ and $\lc{pq}{ij}$ as functions of the coefficients $\delta_j$ and $a_{ij}$.
For any nonnegative integer $\nu$, we define $\poly{\nu}$ to be the set of functions in the following variables:
\begin{itemize}
    \item $\delta_j$ for $ j \leq \nu $
    \item $a_{ij}$ for $ i+j \leq \nu $
\end{itemize}
Often, we will use the notation $\poly{\nu}$ to refer to a function in this set rather than the set itself. For example, $ \delta_4 + \poly{3} $ means a function equal to $ \delta_4 + P_3 $ for some $ P_3 \in \poly{3} $.

We use throughout this section the inclusion $ \poly{0} \subseteq \poly{1} \subseteq \poly{2} \subseteq \poly{3} \subseteq \hdots $ and the fact that the set $\poly{\nu}$ is closed under products and linear combinations.
Note that the class of functions $\poly{0}$ consists only of constants.

Before working with the length spectral data, we analyze how the coefficients in the power series expansion (\ref{hA-triangular}) depend on certain coefficients $\delta_j$ and $a_{ij}$.
We do so by the inductive procedure of substituting an order $\nu$ approximation for $h_A$, $h_B$, $u$, and $\delta$ into (\ref{hA-implicit}) to yield an order $\nu+1$ approximation of $h_A$.
We first need to find the power series of $u$ in terms of that of $M$:

\begin{lemma} \label{lem:u-series}
    Let the power series of $u$ be
    $$ u(h_A,h_B) = \sum_{i,j} u_{ij} h_A^i h_B^j. $$
    Then, for integers $ i \geq 1 $ and $ j \geq 0 $,
    \begin{equation} \label{u-coef-formula}
        u_{i-1,j} = 2i a_{ij} + \poly{i+j-1}.
    \end{equation}
\end{lemma}

\begin{proof}
    We first consider the power series of $ v \circ \Psi^{-1} $,
    $$ (v \circ \Psi)(\eta_A,\eta_B) = \sum_{i,j} \hat{v}_{ij} \eta_A^i \eta_B^j. $$
    Then, $ \xi_A = (v \circ \Psi)(\eta_A,\eta_B) $, so it can be directly computed by (\ref{M-def}) that
    $$ \hat{v}_{i-1,j} = i\hat{a}_{ij}. $$
    We now claim that for all $ (i,j) \neq (0,0) $,
    \begin{equation} \label{a-coef-same}
        \hat{a}_{ij} = \xi_\infty^{i+j} a_{ij} + \poly{i+j-1}
    \end{equation}
    and
    \begin{equation} \label{r-coef-same}
        \hat{v}_{ij} = \xi_\infty^{i+j} v_{ij} + \poly{i+j},
    \end{equation}
    where the power series of $v$ is
    $$ v(h_A,h_B) = \sum_{i,j} v_{ij} h_A^i h_B^j. $$
    By Corollary \ref{cor:q-function}, $ h_A = \eta_A v(h_A,h_B) $ and $ h_B = \eta_B v(h_B,h_A) $, so we may write
    $$ \sum_{i,j} \hat{a}_{ij} \eta_A^i \eta_B^j = M(\eta_A,\eta_B) = \sum_{k,\ell} a_{k\ell} h_A^k h_B^\ell = \sum_{k,\ell} a_{k\ell} \eta_A^k \eta_B^\ell v(h_A,h_B)^k v(h_B,h_A)^\ell. $$
    We may expand $v(h_A,h_B)$ and $v(h_B,h_A)$ into power series in $\eta_A$ and $\eta_B$, and each will have constant term $\xi_\infty$.
    Therefore, in equating the coefficients of $ \eta_A^i \eta_B^j $ on both sides, all terms on the right side where $ k > i $ or $ \ell > j $ vanish and we have (\ref{a-coef-same}), where $ \xi_\infty^{i+j} a_{ij} $ arises from $ (k,\ell) = (i,j) $ and $\poly{i+j-1}$ encompasses all other terms.
    The equation (\ref{r-coef-same}) can be proven similarly.
    
    We have proven that, for integers $ i \geq 1 $ and $ j \geq 0 $,
    $$ v_{i-1,j} = \xi_\infty i a_{ij} + \poly{i+j-1}. $$
    By the formula (\ref{u-def}), for $ (i,j) \neq (0,0) $, we have
    $$ u_{i-1,j} = \frac{1}{\xi_\infty^2} \sum_{k=0}^{i-1} \sum_{\ell=0}^j v_{k\ell} v_{i-1-k,j-\ell} = \sum_{k=0}^{i-1} \sum_{\ell=0}^j (k+1)(i-k) a_{k+1,\ell} a_{i-k,j-\ell} + \poly{i+j-1}. $$
    All terms of this double summation are in $\poly{i+j-1}$ except those where $ (k,\ell) = (0,0) $ or $ (k,\ell) = (i-1,j) $, so this simplifies to (\ref{u-coef-formula}).
\end{proof}

As we found in the proof of Proposition \ref{prop:h-triangular}, $ 1 + u(h_A,h_B) $ and $ \parens{1 + \delta(h_A)}^{2m} $ are both triangular power series, so the formula (\ref{hA-implicit}) determines coefficients of $h_A$ via a quadruple summation, i.e.,
\begin{equation} \label{hA-4-sum}
    \hc{p+k,q+\ell}{k\ell} = \sum_{r=0}^{p-1} \sum_{s=0}^q \sum_{\alpha=0}^k \sum_{\beta=0}^\ell \coef{1 + u(h_A,h_B)}{r+\alpha,s+\beta}{\alpha\beta} \coef{\parens{1 + \delta(h_A)}^{2m}}{(p-1-r)+(k-\alpha), (q-s)+(\ell-\beta)}{k-\alpha, \ell-\beta}.
\end{equation}
We now use this formula to determine how certain coefficients $\hc{p+k,q}{k0}$ depend on certain coefficients in the power series expansions of $\mu$ and $\tilde{M}$.
This will require computations in which we construct the components of this formula and analyze how each component depends on the power series of $\mu$ and $M$.
We consider certain low order coefficients of this power series separately:

\begin{lemma} \label{lem:h-order-1}
    The following hold:
    \begin{itemize}
        \item $ \hc{10}{00} = 1 $
        \item $ \hc{20}{10} = \coef{\parens{1 + \delta(h_A)}^{2m}}{10}{10} = 2\delta_1 $
    \end{itemize}
\end{lemma}

\begin{proof}
    The value $ \hc{10}{00} = 1 $ follows directly from the fact that $ h_A \sim z_A $ for large $m$ and $n$.
    
    As a special case of the formula (\ref{hA-4-sum}),
    $$ \hc{20}{10} = \coef{1 + u(h_A,h_B)}{00}{00} \coef{\parens{1 + \delta(h_A)}^{2m}}{10}{10} + \coef{1 + u(h_A,h_B)}{10}{10} \coef{\parens{1 + \delta(h_A)}^{2m}}{00}{00}. $$
    It is easy to see that $ \coef{1 + u(h_A,h_B)}{00}{00} = \coef{\parens{1 + \delta(h_A)}^{2m}}{00}{00} = 1 $, and since $u(h_A,h_B)$ is a strictly triangular power series, $ \coef{1 + u(h_A,h_B)}{10}{10} = 0 $.
    Therefore, we may write
    $$ \hc{20}{10} = \coef{\parens{1 + \delta(h_A)}^{2m}}{10}{10} = \sum_{i=1}^\infty \coef{\binom{2m}{i} \delta(h_A)^i}{10}{10}. $$
    However, $\delta(h_A)^i$ is a multiple of $z_A^i$, and any multiple of $z_A^i$ for $ i \geq 2 $ must have a coefficient of zero in front of $ mz^A $, so only the term $i=1$ survives, yielding
    $$ \hc{20}{10} = \coef{2m \delta(h_A)}{10}{10} = 2 \coef{\delta(h_A)}{10}{00}. $$
    The result follows from the first-order approximation $ \delta(h_A) = \delta_1 h_A + o(h_A) = \delta_1 z_A + o(z_A) $.
\end{proof}

\begin{lemma} \label{lem:hc}
    For all integers $ p \geq 1 $ and $ q \geq 0 $ with $ (p,q) \neq (1,0) $,
    \begin{equation} \label{hc-pq-00}
        \hc{pq}{00} = 2p a_{pq} + \poly{p+q-1},
    \end{equation}
    and for all integers $ p \geq 2 $,
    \begin{equation} \label{hc-p0-10}
        \hc{p+1,0}{10} = 2\delta_p + 4p(p+1) \delta_1 a_{p0} + \poly{p-1}.
    \end{equation}
\end{lemma}

To prove this lemma, we show that $\hc{pq}{00}$ depends only on $a_{pq}$ and lower order coefficients in (\ref{mu-series}) and (\ref{M-series-h}) and that $\hc{p+1,0}{10}$ depends only on $\delta_p$, $a_{p0}$, and lower order coefficients. These lower order coefficients are encompassed in $\poly{p+q-1}$ and $\poly{p-1}$, respectively. We find explicitly the dependence of these on $\delta_p$ and $a_{p0}$.
We will make frequent use of the following result, which we prove simultaneously:

\begin{lemma} \label{lem:h-pow-order}
    For all integers $ p \geq 1 $, $ q \geq 0 $, $ i \leq p $, and $ j \leq q $,
    \begin{equation} \label{h-pow-order}
        \coef{h_A^i h_B^j}{pq}{00} \in \poly{p+q+1-i-j} \textand \coef{h_A^i h_B^j}{p+1,q}{10} \in \poly{p+q+1-i-j}.
    \end{equation}
\end{lemma}

\begin{proof}[Proof of Lemmas \ref{lem:hc} and \ref{lem:h-pow-order}]
    To prove Lemma \ref{lem:hc}, we prove all of the following statements:
    \begin{enumerate}[label = (\alph*)]
        \item $ \coef{1 + u(h_A,h_B)}{p-1,q}{00} = 2p a_{pq} + \poly{p+q-1} $ for $ p \geq 1 $ and $ q \geq 1 $ with $ (p,q) \neq (0,0) $
        \item $ \coef{1 + u(h_A,h_B)}{pq}{10} = 4p(p-1) \delta_1 a_{pq} + \poly{p+q-1} $ for $ p \geq 1 $ and $ q \geq 1 $ with $ (p,q) \neq (0,0) $
        \item $ \hc{pq}{00} = 2p a_{pq} + \poly{p+q-1} $ for $ p \geq 1 $ and $ q \geq 1 $ with $ (p,q) \neq (0,0) $
        \item $ \coef{\delta(h_A)}{p0}{00} = \delta_p + 2p\delta_1 a_{p0} + \poly{p-1} $ for $ p \geq 2 $
        \item $ \coef{\parens{1 + \delta(h_A)}^{2m}}{p0}{10} = 2\delta_p + 4p\delta_1 a_{p0} + \poly{p-1} $ for $ p \geq 2 $
        \item $ \hc{p+1,0}{10} = 2\delta_p + 4p(p+1)\delta_1 a_{p0} + \poly{p-1} $ for $ p \geq 2 $
    \end{enumerate}
    
    We first prove the base case for Lemma \ref{lem:h-pow-order}: $ (p,q) = (1,0) $. To do so, it suffices to prove that
    $$ \coef{h_A^i h_B^j}{k+1,0}{k0} = \begin{cases}
        1, & (i,j) = (1,0) \text{ and } k=0 \\
        2\delta_1, & (i,j) = (1,0) \text{ and } k=1 \\
        0, & (i,j) \neq (1,0).
    \end{cases} $$
    Both values for $ (i,j) = (0,0) $ are direct consequences of Lemma \ref{lem:h-order-1}.
    If $ j \geq 1 $, then $ h_A^i h_B^j $ is a multiple of $z_B^j$, from which $ \coef{h_A^i h_B^j}{k+1,0}{k0} = 0 $ is a direct result.
    If $ i \geq 2 $ and $j=0$, then since $\parens{\frac{h_A}{z_A}}^i$ is a triangular power series, we have $ \coef{h_A^i}{k+1,0}{k0} = \coef{\parens{\frac{h_A}{z_A}}^i}{k+1-i,0}{k0} $, which is zero since $ k+1-i < k $.
    This completes the proof of Lemma \ref{lem:h-pow-order} when $ (p,q) = (1,0) $.
    
    Now consider any integer $ \nu \geq 2 $, and assume that (a)-(c) are true for $ p+q \leq \nu-1 $, (d)-(f) are true for $ p \leq \nu-1 $, and Lemma \ref{lem:h-pow-order} is true for $ p+q \leq \nu-1 $. Note that in particular, our inductive hypothesis implies that all quantities in (a)-(f) for such $p$ and $q$ are in $\poly{\nu-1}$.
    We will prove that (a)-(c) are true for $ p+q = \nu $, (d)-(f) are true for $ p = \nu $, and Lemma \ref{lem:h-pow-order} holds for $ p+q = \nu $.
    
    Let $ p \geq 1 $ and $ q \geq 0 $ be such that $ p+q = \nu $.
    We expand
    $$ \coef{1 + u(h_A,h_B)}{p-1+k,q}{k0} = \sum_{i,j} u_{ij} \coef{h_A^i h_B^j}{p-1+k,q}{k0}. $$
    The quantity $ h_A^i h_B^j $ is a multiple of $ z_A^i z_B^j $, so if $ i > p-1+k $ or $ j > q $, then its coefficient in front of $ m^k z_A^{p-1+k} z_B^q $ must be zero. This allows us to restrict the sum to $ i \leq p-1+k $ and $ j \leq q $.
    Also, $ \coef{h_A^i h_B^j}{p-1+k,q}{k0} = \coef{\parens{\frac{h_A}{z_A}}^i \parens{\frac{h_B}{z_B}}^j}{p-1+k-i,q-j}{k0} $ and $\frac{h_A}{z_A}$ and $\frac{h_B}{z_B}$ are triangular power series, so this coefficient can only be nonzero if $ p-1+k-i \geq k $, or equivalently, if $ i \leq p-1 $. We may then further restrict this sum to
    $$ \coef{1 + u(h_A,h_B)}{p-1+k,q}{k0} = \sum_{i=0}^{p-1} \sum_{j=0}^q u_{ij} \coef{h_A^i h_B^j}{p-1+k,q}{k0}. $$
    By Lemma \ref{lem:h-pow-order}, $ \coef{h_A^i h_B^j}{p-1+k,q}{k0} \in \poly{\nu-1} $ for all terms in this sum, and by definition $ u_{ij} \in \poly{i+j+1} \subseteq \poly{\nu-1} $ for all terms except that where $ (i,j) = (p-1,q) $. Therefore, applying Lemma \ref{lem:u-series},
    $$ \coef{1 + u(h_A,h_B)}{p-1+k,q}{k0} = u_{p-1,q} \coef{h_A^{p-1} h_B^q}{p-1+k,q}{k0} + \poly{\nu-1} = 2p a_{pq} \coef{h_A^{p-1} h_B^q}{p-1+k,q}{k0} + \poly{\nu-1}. $$
    For $k=0$, the result $ \coef{h_A^{p-1} h_B^q}{p-1,q}{00} = 1 $ follows from the fact that $ h_A \sim z_A $ and $ h_B \sim z_B $, proving (a). Using Lemma \ref{lem:h-order-1}, we may compute
    $$ \coef{h_A^{p-1} h_B^q}{pq}{10} = (p-1) \coef{h_A}{20}{10} \parens{\coef{h_A}{10}{00}}^{p-2} \parens{\coef{h_B}{01}{00}}^q = 2(p-1)\delta_1, $$
    which proves (b).
    
    Now consider the statement (c). When $k=0$ and $\ell=0$, the decomposition (\ref{hA-4-sum}) takes the form
    $$ \hc{pq}{00} = \sum_{r=0}^{p-1} \sum_{s=0}^q \coef{1 + u(h_A,h_B)}{rs}{00} \coef{\parens{1 + \delta(h_A)}^{2m}}{p-1-r, q-s}{00}. $$
    However, if we expand
    $$ \parens{1 + \delta(h_A)}^{2m} = 1 + \sum_{i=1}^\infty \binom{2m}{i} \delta(h_A)^i, $$
    then $m$ will be a factor of $\binom{2m}{i}$ for all $i$, so
    \begin{equation} \label{delta-fac-rs-00}
        \coef{\parens{1 + \delta(h_A)}^{2m}}{rs}{00} = \begin{cases} 1, & (r,s) = (0,0) \\ 0, & (r,s) \neq (0,0). \end{cases}
    \end{equation}
    Therefore, only one term in the above double summation for $\hc{pq}{00}$ survives, and $ \hc{pq}{00} = \coef{1 + u(h_A,h_B)}{p-1,q}{00} $. The statement (c) then follows from (a).
    
    Now let $p=\nu$.
    To prove (d), expand
    $$ \coef{\delta(h_A)}{p0}{00} = \sum_{j=1}^p \delta_j \coef{h_A^j}{p0}{00}. $$
    By Lemma \ref{lem:h-pow-order}, $ \coef{h_A^j}{p0}{00} \in \poly{p-1} $ for $ j \geq 2 $, and by definition, $ \delta_j \in \poly{p-1} $ for $ j \leq p-1 $. Therefore,
    $$ \coef{\delta(h_A)}{p0}{00} = \delta_1 \coef{h_A}{p0}{00} + \delta_p \coef{h_A^p}{p0}{00} + \poly{p-1} = \delta_p + \delta_1 \hc{p0}{00} + \poly{p-1}. $$
    The statement (d) then follows from (c).
    
    To prove (e), first expand
    \begin{equation} \label{hc-pf-e}
        \coef{\parens{1 + \delta(h_A)}^{2m}}{p0}{10} = \sum_{i=1}^p \coef{\binom{2m}{i} \delta(h_A)^i}{p0}{10},
    \end{equation}
    where the upper bound of $p$ is due to the fact that $\delta(h_A)^i$ is a multiple of $z_A^i$, and hence the coefficient in front of $ m z_A^p $ of any multiple of $\delta(h_A)^i$ must be zero whenever $ i > p $.
    The binomial coefficient $\binom{2m}{i}$ is a polynomial in $m$ with zero constant term, so if $\coef{\binom{2m}{i}}{00}{10}$ denotes its coefficient in front of $m$, then
    $$ \coef{\binom{2m}{i} \delta(h_A)^i}{p0}{10} = \coef{\binom{2m}{i}}{00}{10} \coef{\delta(h_A)^i}{p0}{00} = \frac{2 (-1)^{i-1}}{i} \coef{\delta(h_A)^i}{p0}{00}. $$
    By Lemma \ref{lem:h-pow-order}, $ \coef{\delta(h_A)^i}{p0}{00} \in \poly{p-1} $ for $ i \geq 2 $, so only the term $i=1$ in (\ref{hc-pf-e}) survives and
    $$ \coef{\parens{1 + \delta(h_A)}^{2m}}{p0}{10} = 2 \coef{\delta(h_A)}{p0}{00}. $$
    The statement (e) then follows from (d).
    
    To prove (f), use a special case of the quadruple summation formula (\ref{hA-4-sum}):
    \begin{align*}
        \hc{p+1,0}{10} & = \sum_{r=0}^{p-1} \coef{1 + u(h_A,h_B)}{r+1,0}{10} \coef{\parens{1 + \delta(h_A,h_B)}^{2m}}{p-1-r,0}{00} \\
        & \qquad + \sum_{r=0}^{p-1} \coef{1 + u(h_A,h_B)}{r0}{00} \coef{\parens{1 + \delta(h_A,h_B)}^{2m}}{p-r,0}{10}.
    \end{align*}
    By (\ref{delta-fac-rs-00}), all terms in the first summation vanish except for that where $ r = p-1 $. In addition, by statements (a) and (e) in the inductive hypothesis, all terms in the second double summation are in $\poly{p-1}$ except those where $ r = 0 $ or $ r = p-1 $.
    Therefore,
    $$ \hc{p+1,0}{10} = \coef{1 + u(h_A,h_B)}{p0}{10} + \coef{\parens{1 + \delta(h_A,h_B)}^{2m}}{p0}{10} + \coef{\parens{1 + \delta(h_A,h_B)}^{2m}}{10}{10} \coef{1 + u(h_A,h_B)}{p-1,0}{00} + \poly{p-1}. $$
    The statement (f) then follows from Lemma \ref{lem:h-order-1}, (a), (b), and (e).
    
    To prove Lemma \ref{lem:h-pow-order} for $ p+q = \nu $, expand
    $$ \coef{h_A^i h_B^j}{p+k,q}{k0} = \sum \hc{p_1+k_1,q_1}{k_1 0} \hc{p_2+k_2,q_2}{k_2 0} \hdots \hc{p_i+k_i,q_i}{k_i 0} \hc{q_{i+1}, p_{i+1} + k_{i+1}}{0 k_{i+1}} \hc{q_{i+2}, p_{i+2} + k_{i+2}}{0 k_{i+2}} \hdots \hc{q_{i+j}, p_{i+j} + k_{i+j}}{0 k_{i+j}}, $$
    where $ k \in \set{0,1} $ and the sum ranges over all tuples of nonnegative integers $(p_1,p_2,\hdots,p_{i+j})$, $(q_1,q_2,\hdots,q_{i+j})$, and $(k_1,k_2,\hdots,k_{i+j})$ where $ p_1 + p_2 + \hdots + p_{i+j} = p $, $ q_1 + q_2 + \hdots + q_{i+j} = q $, and $ k_1 + k_2 + \hdots + k_{i+j} = k $.
    Since $\frac{h_A}{z_A}$ is a triangular power series, we may impose the restrictions that $ p_\mu \geq 1 $ for all $ 1 \leq \mu \leq i $ and $ q_\mu \geq 1 $ for all $ i+1 \leq \mu \leq j $. This implies that $ p_\mu + q_\mu \geq 1 $ for all $\mu$, and hence that $ p_\mu + q_\mu \leq \nu+1-i-j $ for all $\mu$.
    By the cases in Lemma \ref{lem:hc} which we have already proven, this implies $ \hc{p_\mu + k_\mu, q_\mu}{k_\mu 0} \in \poly{\nu+1-i-j} $ for all $ 1 \leq \mu \leq i $ and $ \hc{q_\mu, p_\mu + k_\mu}{0 k_\mu} \in \poly{\nu+1-i-j} $ for all $ i+1 \leq \mu \leq i+j $. Hence, $ \coef{h_A^i h_B^j}{p+k,q}{k0} \in \poly{\nu+1-i-j} $.
\end{proof}

We are now ready to apply Proposition \ref{prop:length-action} to investigate the coefficients in the power series (\ref{length-triangular}) of the length spectral data. To do so, we first need to investigate the function $\Sigma$:

\begin{lemma} \label{lem:Sigma-delta}
    It holds that $ \Sigma(h) = \Oo(h^2) $, and if the power series of $\Sigma$ is expressed as
    \begin{equation} \label{Sigma-series}
        \Sigma(h) = \sum_{j=1}^\infty \hat{\delta}_j h^{j+1},
    \end{equation}
    then
    \begin{equation} \label{delta-hat-1}
        \hat{\delta}_1 = \frac{1}{2} \delta_1
    \end{equation}
    and for all $ j \geq 2 $,
    \begin{equation} \label{Sigma-delta}
        \hat{\delta}_j = \frac{j}{j+1} \delta_j + \poly{j-1}.
    \end{equation}
\end{lemma}

\begin{proof}
    Substituting the first order approximation $ \mu(h) = \lambda \, \parens{1 + \delta_1 h + \Oo(h^2)} $ into (\ref{Sigma-def}) yields $ \Sigma(h) = \frac{1}{2} \delta_1 h^2 + \Oo(h^3) $. Hence, $ \Sigma(h) = \Oo(h^2) $ and (\ref{delta-hat-1}) is true.
    
    Now consider $ j \geq 2 $. We may compute from (\ref{Sigma-def}) that
    $$ \hat{\delta}_j = \frac{\Sigma^{(j+1)}(0)}{(j+1)!} = \eval{\frac{1}{j!} \frac{d^j}{\dd h^j} \parens{\log \mu(h)} - \frac{1}{(j+1)!} \frac{d^j}{\dd h^j} \parens{\int_0^h \log \mu(\tau) \dd\tau}}_{h=0} = \eval{\frac{j}{(j+1)!} \frac{d^j}{\dd h^j} \log \mu(h)}_{h=0}. $$
    We may directly compute
    $$ \frac{d^j}{\dd h^j} \log \mu(h) = \frac{d^{j-1}}{\dd h^{j-1}} \parens{\frac{\mu'(h)}{\mu(h)}} = \frac{\mu^{(j)}(h)}{\mu(h)} + D_{j-1}\parens{\mu(h), \mu'(h), \mu''(h), \hdots, \mu^{(j-1)}(h)} $$
    for some function $D_{j-1}$. In particular, $ \mu^{(k)}(0) = k! \lambda \delta_k $ for all $ k \geq 1 $ and $ \mu(0) = \lambda $, so
    $$ \frac{d^j}{\dd h^j} \log \mu(h) = j! \delta_j + \poly{j-1}, $$
    and hence, (\ref{Sigma-delta}) is true.
\end{proof}

We are now ready to determine coefficients in (\ref{length-triangular}).

\begin{lemma} \label{lem:l-order-1}
    It holds that $ \lc{20}{10} = -\delta_1 $.
\end{lemma}

\begin{proof}
    By the formula (\ref{length-action}),
    $$ \lc{20}{10} = 2 \coef{\Sigma(h_A)}{20}{00} + \coef{h_A}{20}{10} + \coef{h_B}{20}{10} - 2 \coef{\tilde{M}(h_A,h_B)}{20}{10}. $$
    We have $ \tilde{M}(h_A,h_B) = h_A + a_{20} h_A^2 + \Oo(h_A^3 + h_B) $, so $ \coef{\tilde{M}(h_A,h_B)}{20}{10} = \coef{h_A}{20}{10} + a_{20} \coef{h_A^2}{20}{10} $, which implies
    $$ \lc{20}{10} = 2 \coef{\Sigma(h_A)}{20}{00} - \coef{h_A}{20}{10} + \coef{h_B}{20}{10} - 2a_{20} \coef{h_A^2}{20}{10}. $$
    By Lemma \ref{lem:Sigma-delta}, $ \Sigma(h_A) = \frac{1}{2} \delta_1 h_A^2 + \Oo(h_A^3) = \frac{1}{2} \delta_1 z_A^2 + o(z_A^2) $, so $ \coef{\Sigma(h_A)}{20}{00} = \frac{1}{2}\delta_1 $.
    By Lemma \ref{lem:h-order-1}, $ \coef{h_A}{20}{10} = 2\delta_1 $, and clearly $ \coef{h_A^2}{20}{10} = \coef{h_B}{20}{10} = 0 $.
\end{proof}

\begin{lemma} \label{lem:lc}
    For all nonnegative integers $p$ and $q$ with $ p+q \geq 2 $,
    \begin{equation} \label{lc-pq-00}
        \lc{pq}{00} = -2(p+q+1) a_{pq} + \poly{p+q-1},
    \end{equation}
    and for all integers $ p \geq 2 $,
    \begin{equation} \label{lc-p0-10}
        \lc{p+1,0}{10} = -\frac{2}{p+1} \delta_p - 4p(p+1)\delta_1 a_{p0} + \poly{p-1}.
    \end{equation}
\end{lemma}

\begin{proof}
    First, notice that
    $$ \coef{\tilde{M}(h_A,h_B)}{pq}{00} = \sum_{i=0}^p \sum_{j=0}^q a_{ij} \coef{h_A^i h_B^j}{pq}{00}. $$
    By Lemma \ref{lem:h-pow-order},$ \coef{h_A^i h_B^j}{pq}{00} \in \poly{p+q-1} $ for all terms in this summation except those where $ (i,j) = (1,0) $ or $ (i,j) = (0,1) $, and the coefficient $a_{ij}$ itself is in $\poly{p+q-1}$ for all $(i,j)$ except $(p,q)$. Therefore,
    $$ \coef{\tilde{M}(h_A,h_B)}{pq}{00} = a_{10} \coef{h_A}{pq}{00} + a_{01} \coef{h_B}{pq}{00} + a_{pq} \coef{h_A^p h_B^q}{pq}{00} + \poly{p+q-1}, $$
    and using the values $ a_{10} = a_{01} = \coef{h_A^p h_B^q}{pq}{00} = 1 $,
    $$ \coef{\tilde{M}(h_A,h_B)}{pq}{00} = \hc{pq}{00} + \hc{qp}{00} - 2a_{pq} + \poly{p+q-1}. $$
    We then compute
    $$ \lc{pq}{00} = \coef{h_A}{pq}{00} + \coef{h_B}{pq}{00} - 2\coef{\tilde{M}(h_A,h_B)}{pq}{00} = -\hc{pq}{00} - \hc{qp}{00} - 2a_{pq}. $$
    The result (\ref{lc-pq-00}) follows from Lemma \ref{lem:hc} and the fact that $ a_{pq} = a_{qp} $ by symmetry.
    
    By taking the coefficient in front of $ m z_A^{p+1} z_B^q $ in (\ref{length-action}), we have
    \begin{equation} \label{lc-pf-2}
        \lc{p+1,0}{10} = 2 \coef{\Sigma(h_A)}{p+1,0}{00} + \coef{h_A}{p+1,0}{10} - 2 \coef{\tilde{M}(h_A,h_B)}{p+1,0}{10}.
    \end{equation}
    By similar arguments as above,
    $$ \coef{\tilde{M}(h_A,h_B)}{p+1,0}{10} = \hc{p+1,0}{10} + a_{p0} \coef{h_A^p}{p+1,0}{10} + \poly{p-1}. $$
    It was computed in the proof of (b) in Lemma \ref{lem:hc} that $ \coef{h_A^p}{p+1,0}{10} = 2p\delta_1 $, so by this and the value of $\hc{p+1,0}{10}$ in Lemma \ref{lem:hc}, we have
    \begin{equation} \label{lc-pf-2a}
        \coef{\tilde{M}(h_A,h_B)}{p+1,0}{10} = 2\delta_p + 2p(2p+3)\delta_1 a_{p0} + \poly{p-1}.
    \end{equation}
    We now expand
    $$ \coef{\Sigma(h_A)}{p+1,0}{00} = \sum_{j=1}^p \hat{\delta}_p \coef{h_A^{j+1}}{p+1,0}{00}. $$
    By Lemma \ref{lem:h-pow-order}, all coefficients $ \coef{h_A^{j+1}}{p+1,0}{00} $ in this summation are in $\poly{p-1}$ except that for $j=1$. Also, by Lemma \ref{lem:Sigma-delta}, all coefficents $\hat{\delta}_j$ in this summation are in $\poly{p-1}$ except $\hat{\delta}_p$. Therefore,
    $$ \coef{\Sigma(h_A)}{p+1,0}{00} = \hat{\delta}_1 \coef{h_A^2}{p+1,0}{00} + \hat{\delta}_p \coef{h_A^{p+1}}{p+1,0}{00} + \poly{p-1}. $$
    By expanding
    $$ \coef{h_A^2}{p+1,0}{00} = \sum_{k=1}^p \hc{k0}{00} \hc{p+1-k,0}{00} = 2\hc{p0}{00} + \poly{p-1} $$
    and applying Lemma \ref{lem:Sigma-delta},
    \begin{equation} \label{lc-pf-2b}
        \coef{\Sigma(h_A)}{p+1,0}{00} = \delta_1 \hc{p0}{00} + \frac{p}{p+1} \delta_p + \poly{p-1} = \frac{p}{p+1} \delta_p + 2p\delta_1 a_{p0} + \poly{p-1}.
    \end{equation}
    Substituting the value of $ \coef{h_A}{p+1,0}{10} = \hc{p+1,0}{10} $ in Lemma \ref{lem:hc} and the values (\ref{lc-pf-2a}) and (\ref{lc-pf-2b}) into (\ref{lc-pf-2}) completes the proof of (\ref{lc-p0-10}).
\end{proof}

We now have all of the elements needed to prove our main theorems.

\begin{proof}[Proof of Theorem \ref{th:conj}]
    The results of Section \ref{sec:Birkhoff} show that two systems are analytically conjugate to each other near the homoclinic orbit if and only if they have the same Birkhoff normal form $N$ and the same gluing map $G$. Therefore, it suffices to prove that $N$ and $G$ are uniquely determined by the length spectral data. We do so by recovering $\lambda$, $\xi_\infty$, and the entire power series of the functions $\mu$ and $M$ from the length spectral data.

    First, notice that by Corollary \ref{cor:l-UT}, the length spectral data can be expressed as
    $$ \ell_{m,n} = (2m+2n) \ell_0 + \sum_{p,q} \sum_{i=0}^{\max\set{0,p-1}} \sum_{j=0}^{\max\set{0,q-1}} \xi_\infty^{2(p+q)} \lc{pq}{ij} m^i n^j \parens{\lambda^{2m}}^p \parens{\lambda^{2n}}^q. $$
    Hence, $\lambda$ and all of the quantities $ \xi_\infty^{2(p+q)} \lc{pq}{ij} $ can be recovered from the length spectral data.
    By Proposition \ref{prop:l-order-1}, the first order approximation is
    $$ \ell_{m,n} = (2m+2n) \ell_0 + 2L_\infty - \xi_\infty^2 \lambda^{2m} - \xi_\infty^2 \lambda^{2n} + o(\lambda^{2m} + \lambda^{2n}), $$
    so the quantity $\xi_\infty^2$ can be recovered, and therefore, so can $\lc{pq}{ij}$ for every $p$, $q$, $i$, and $j$.
    
    It remains is to show that all coefficients $\delta_j$ and $a_{ij}$ can be recovered from the coefficients $\lc{pq}{ij}$.
    First off, we can recover $\delta_1$ from the formula in Lemma \ref{lem:l-order-1}, and the first order approximation (\ref{M-order-1}) gives the values $ a_{10} = a_{01} = 1 $.
    
    Consider a positive integer $ \nu \geq 2 $.
    Assuming we know the values of $\delta_j$ for $ j \leq \nu-1 $ and of $a_{ij}$ for $ i+j \leq \nu-1 $, any function in $\poly{\nu-1}$ has a known value.
    Lemma \ref{lem:lc} then allows us to solve for the values of $\delta_\nu$ and $a_{pq}$ for $ p+q = \nu $ given the coefficients $\lc{p+1,0}{10}$ and $\lc{pq}{00}$ for $ p+q \leq \nu $.
    Hence, all coefficients $\delta_j$ and $a_{ij}$ can be recovered by induction on $\nu$.
    
    Finally, all parts of the formula in Proposition \ref{length-action} can be determined only by the functions $M$ and $\mu$, so if two systems have the same normal form $N$ and gluing map $G$, then their length spectral data $\ell_{m,n}$ will be identical.
\end{proof}

\begin{proof}[Proof of Theorem \ref{th:fixed}]
    The regions $D_1$ and $D_2$ determine the collision maps $\cmap_{12}$ and $\cmap_{21}$, which in turn determine the normal form $N$ and conjugacies $\Phi_1$ and $\Phi_2$ by Proposition \ref{prop:BNF}.
    
    By Theorem \ref{th:conj}, from the length spectral data $\ell_{m,n}$, we can recover the Birkhoff coordinates $ \parens{\xi_k^{m,n}, \eta_k^{m,n}} $ of each point $x_k^{m,n}$ in the cyclicity-$2$ orbits for large enough $m$ and $n$.
    However, since we know $\Phi_1$ and $\Phi_2$, we can recover exactly the points $x_k^{m,n}$ on these orbits that are in $\cspace_1$ or $\cspace_2$.
    In particular, we may recover all of the points $x_k^n$ in the cyclicity-$1$ orbits for large enough $n$ and $ 1 \leq k \leq n-1 $.
    
    If $ x_k^n = (i_k^n,s_k^n,r_k^n) $, then we can recover the point of collision $\gamma_1(s_1^n)$, the unit normal $\vec{\nu}$ to scatterer $1$ at this point, and the angle of incidence $ \varphi_1^n = \sin^{-1}(r_1^n) $. If $\vec{u}$ is the rotation of $\vec{\nu}$ by an angle of $\varphi_1^n$, then the ray from the collision point $\gamma_1(s_1^n)$ to the collision point $\gamma_3(s_0^n)$ is in the direction of $\vec{u}$.
    Also, since $ s_{n-1}^n = s_1^n $, the perimeter of this orbit is exactly
    $$ \ell_n = 2L_{31}\parens{s_0^n, s_1^n} + \sum_{k=2}^{n-1} L_{i_{k-1} i_k}\parens{s_{k-1}^n, s_k^n}. $$
    However, we know $\ell_n$ and all of the distances $ L_{i_{k-1} i_k}\parens{s_{k-1}^n, s_k^n} $ for $ 2 \leq k \leq n-1 $, so we can recover the distance $ L_{31}\parens{s_0^n, s_1^n} $.
    We then have $ \gamma_3(s_0^n) = \gamma_1(s_1^n) + L_{31}\parens{s_0^n, s_1^n} \vec{u} $.
    This geometric argument is illustrated in Figure \ref{fig:fixed-pf}.
    
    Using this process, we can find a sequence of points $\gamma_3(s_0^n)$ on $ \partial D_3 $. These points will accumulate at the homoclinic point $\gamma_3(s_0^\infty)$, so there is a unique analytic curve that passes through all of these points and $\gamma_3(s_0^\infty)$.
    Since the boundary $ \partial D_3 $ is closed and analytic, it must be the unique analytic continuation of this curve.
    This determines the region $D_3$.
\end{proof}

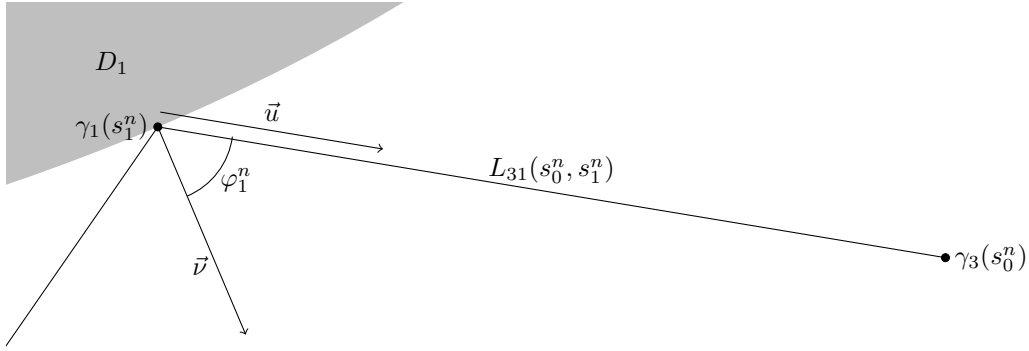
\begin{figure}[ht]
    \centering
    \begin{tikzpicture}[x=2cm, y=2cm]
        \clip (6.3,2.2) rectangle (13.3,4.7);
        \fill[black!25] (0,9.4) ellipse (12.2 and 6.9);
        \fill[black!25] (0,-13.9) ellipse (13.3 and 11.4);
        \node at (7,4.3) {$D_1$};
        
        \draw[fill] (12.5417, 3.0041) circle[radius=1.5pt];
		\node[right](A0) at (12.5417, 3.0041) {$\gamma_3(s_0^n)$};
		\draw[fill] (7.3062, 3.8742) circle[radius=1.5pt];
		\node[left](A1) at (7.3062, 3.8742) {$\gamma_1(s_1^n)$};
		\node(A2) at (2.7233, -2.7415) {};
		\node[left](B1) at (7.6958, 2.9532) {$\vec{\nu}$};
		\node[right](phi) at (7.6598, 3.5206) {$\varphi_1^n$};
		
		\path[-] (12.5417, 3.0041) edge node[above]{$L_{31}(s_0^n,s_1^n)$} (7.3062, 3.8742);
		\draw (7.3062, 3.8742) -- (2.7233, -2.7415);
		\path[->] (7.3062, 3.8742) edge (7.8905, 2.4927);
		\draw (7.5010, 3.4137) arc[start angle = 292.93, end angle = 352.56, radius = .5];
		\path[->] (7.3226, 3.9728) edge node[above]{$\vec{u}$} (8.8023, 3.7269);
    \end{tikzpicture}
    \caption{The process of determining the collision point $\gamma_3(s_0^n)$ from the point $\gamma_1(s_1^n)$, the unit normal $\vec{\nu}$ to $D_1$, and the angle of incidence $\varphi_1^n$.}
    \label{fig:fixed-pf}
\end{figure}

\section{Conclusion and Possible Generalizations} \label{sec:conclusion}

Theorem \ref{th:conj} reduces the problem of marked length spectrum rigidity of $3$-scatterer systems considered in this paper to a problem of analytic conjugacy.
In order to prove length spectrum rigidity of this class of systems, it suffices to prove the following:

\begin{conjecture} \label{conj:rigid}
    Any two $3$-scatterer systems that are analytically conjugate to one another in a neighborhood of the homoclinic orbit must be isometric, or equivalently, given any normal form $N$ and gluing map $G$, there is at most one triple of functions $(\Phi_1,\Phi_2,\Phi_-)$ that can be realized from the collision map of a $3$-scatterer system.
\end{conjecture}

In general, the order $2n+1$ approximation of the normal form $N$ near the origin has $n$ degrees of freedom. However, the order $2n+1$ approximations of $\cmap_{12}$ and $\cmap_{21}$ near the respective $2$-periodic points depend on the order $2n+2$ approximations of both $ \partial D_1 $ and $ \partial D_2 $ near these collision points. In this sense, the normal form provides less information than the shapes of the scatterers $D_1$ and $D_2$, so $N$ cannot be recovered from the region $ D_1 \cup D_2 $, even if we assume one axis of reflexive symmetry.
On the other hand, there are $\Oo(n^2)$ degrees of freedom for the order $n$ approximation of an area-preserving map $G$ for which $ G \circ I $ is an involution, so the requirement that $G$ remain the same in two systems seems quite restrictive, and it is reasonable to expect that generically a non-isometric system with the same $N$ and $G$ does not exist.
However, determining the gluing map $G$ involves working with the behavior of the system near all points in the homoclinic orbit simultaneously. The approach in this paper of recovering local power series approximations one order at a time is not well-suited for relating the local behavior at one point to the local behavior at another point, so we expect that new methods are needed to prove Conjecture \ref{conj:rigid}.

In this paper, we considered only orbits shadowing the homoclinic orbit with coding $ (21)^\infty3(12)^\infty $. However, we expect the results in this paper to be generalizable to other types of periodic orbits as well.
For example, we may also consider the heteroclinic point on scatterer $3$ with coding $ (12)^\infty3(12)^\infty $ and the cyclicity-$2$ orbits with coding $ 3(12)^m3(12)^n $ or $ 3(12)^{m-1}13(21)^{n-1}2 $.
More generally, given any coding sequence $ \sigma = (i_0,i_1,\hdots,i_{n-1}) \in \Sigma $ of a periodic orbit, we may consider periodic orbits shadowing homoclinic orbits with codings such as $ \bar{\sigma}^\infty i \sigma^\infty $ or $ \sigma^\infty i \sigma^\infty $, where $ \bar{\sigma} = (i_{n-1},\hdots,i_1,i_0) $.
We expect the set of all periodic orbits to be dense in the non-escaping Cantor set. This would lead to the following generalization of Theorem \ref{th:conj}:

\begin{conjecture} \label{conj:general}
    Two $3$-scatterer dispersing billiard systems have the same marked length spectrum if and only if there is an analytic conjugacy between their collision maps defined in some neighborhood of their non-escaping Cantor sets.
\end{conjecture}

If Conjecture \ref{conj:general} is true, then to prove marked length spectrum rigidity, it suffices to prove that two $3$-scatterer systems that are analytically conjugate to each other in some neighborhood of their non-escaping Cantor sets must have isometric billiard tables.

Another possible generalization of our results would be to allow for isometric transformations of the scatterers individually and assume knowledge of the length spectrum as a function of the position of the scatterers.
It is conceivable that one could gather useful information from this data on how the behavior of the system changes near the $2$-periodic points and/or the homoclinic point on scatterer $3$ and that this information could be enough to recover local approximations of the boundaries of the respective scatterers at these points.
We pose the following conjecture:

\begin{conjecture} \label{conj:transformation}
    Let $SE(2)$ be the special Euclidean group, or the Lie group of direct isometries of $\R^2$, and $S$ be a submanifold of $ SE(2)^3 / \Delta $ where $ \Delta := \set{(g,g,g): g \in SE(2)} $. Let $S$ act on the class of $3$-scatterer dispersing billiard systems by $ (g_1,g_2,g_3)(D_1,D_2,D_3) = \parens{g_1(D_1), g_2(D_2), g_3(D_3)} $ and $\ell_{m,n}(g)$ denote the length spectral data $\ell_{m,n}$ of the system $g(D_1,D_2,D_3)$.
    Then, the collection of length spectral data $ \set{\ell_{m,n}(g)}_{m \in \N, n \in \N, g \in S} $ uniquely determine the scatterers $D_1$, $D_2$, and $D_3$.
\end{conjecture}

Modding out by $\Delta$ in this conjecture is due to the fact that applying an isometry to the entire system will not change the length spectrum.
Different submanifolds $S$ provide different variations of this conjecture. For example, we can restrict to only translations and/or allow transformations of only one scatterer.
When $S$ contains only the identity, this conjecture is identical to Conjecture \ref{conj:rigid}.

\section{Acknowledgements}

This paper profited considerably from almost two years of regular discussion on the subject with Professor Vadim Kaloshin.

This material is based upon work supported by the National Science Foundation Graduate Research Fellowship Program under Grant No. DGE 1840340. Any opinions, findings, and conclusions or recommendations expressed in this material are those of the author and do not necessarily reflect the views of the National Science Foundation.

\end{document}